\documentclass[11pt]{article}%
\usepackage{geometry}
\usepackage{amsmath}
\usepackage{amsfonts}
\usepackage{amssymb}
\usepackage{graphicx}%
\usepackage{xcolor}
\usepackage{float}
\providecommand{\U}[1]{\protect\rule{.1in}{.1in}}
\newtheorem{theorem}{Theorem}

\newtheorem{assumption}[theorem]{Assumption}

\newtheorem{corollary}[theorem]{Corollary}

\newtheorem{definition}[theorem]{Definition}

\newtheorem{lemma}[theorem]{Lemma}

\newtheorem{proposition}[theorem]{Proposition}
\newtheorem{remark}[theorem]{Remark}

\newenvironment{proof}[1][Proof]{\noindent\textbf{#1.} }{\ \rule{0.5em}{0.5em}}

\renewcommand{\P}{\mathbf{P}}

\renewcommand{\tilde}{\widetilde}

\renewcommand {\epsilon}{\varepsilon}

\newcommand{\EE}{\mathbb{E}}

\newcommand{\HH}{\mathbb{H}}

\newcommand{\LL}{\mathbb{L}}

\newcommand{\NN}{\mathbb{N}}

\newcommand{\PP}{\mathbb{P}}

\newcommand{\RR}{\mathbb{R}}

\newcommand{\fF}{\mathcal{F}}

\usepackage[colorlinks=true,linkcolor=blue,citecolor=magenta]{hyperref}

\begin{document}

\title{Non-local Conservation Law  from Stochastic  Particle Systems}
\author{Marielle Simon\thanks{INRIA Lille - Nord Europe, MEPHYSTO Team, 
France \emph{and} Laboratoire Paul Painlev\'e, CNRS UMR 8524. E-mail: \textsl{marielle.simon@inria.fr}} , Christian
Olivera\thanks{Departamento de Matem\'{a}tica, Universidade Estadual de
Campinas, Brazil. E-mail: \textsl{colivera@ime.unicamp.br}.}}
\date{}
\maketitle

\begin{abstract}
 We consider an interacting particle system in $\mathbb{R}^d$ modelled as a system of $N$ stochastic differential equations.  The limiting behaviour as the size $N$ grows to infinity is achieved as a
law of large numbers for the empirical density process associated with the interacting particle system.
\end{abstract}


\vspace{0.3cm} \noindent {\bf MSC2010 subject classification:} 60H20, 60H10, 60F99.

\section{Introduction\label{Intro}}

\subsection{Context}

There is a vast and growing interest in modeling systems of large (though still finite) population of individuals subject to mutual interaction and random dispersal (due to, for instance, the environment). We refer the reader to \cite{capasso} for a recent textbook on the subject.  More precisely, the behavior of such systems is often described as the limit of the number of individuals tends to infinity. While at the \emph{microscopic} scale, the population is well modeled by stochastic differential equations (SDEs),  the  \emph{macroscopic} description of the population densities is provided by partial differential equations (PDEs), which can be of different types, for instance linear PDEs for Black-Scholes models, or non-linear PDEs for density-dependent diffusions. All these systems may characterize the collective behavior of individuals in biology models, but also agents in economics and finance. The range of application of this area is huge. 

In the present paper the limit processes that we want to obtain belong to the family of non-local PDEs, which in our case are related to anomalous diffusions. For that purpose, we study the asymptotic behaviour of a system of particles
which interact \emph{moderately}, i.e.~an intermediate situation between weak and strong
interaction, and which are submitted to random scattering. It is well-known that in the case where the particles interact with only few others (which is often the most realistic case), the results can be qualitatively different, but also mathematically more challenging. Nowadays, there are few rigorous results at disposal, starting from the seminal paper \cite{Var}.

The Lagrangian description of our dynamics of the moderately-interacting particles 
is given via  a  system of stochastic differential equations. Suppose that
 for each $i\in\mathbb{N}$, the process $X_{t}^{i,N}$ satisfies the system of coupled stochastic differential equations in $\mathbb{R}^{d}$
\begin{equation}
dX_{t}^{i,N}=F\bigg(X_{t}^{i,N},\frac{1}{N}\sum_{k=1}^{N}V^{N}\big(X_{t}^{i,N}%
-X_{t}^{k,N}\big)\bigg)\ dt+dL_{t}^{i},\label{itoass}%
\end{equation}
where $\{L_{t}^{i}\}_{i\in\mathbb{N}}$ are independent $\mathbb{R}^{d}$-valued symmetric $\alpha$-stable   L\'{e}vy processes on a filtered probability space $\left(
\Omega,\mathcal{F},\mathcal{F}_{t},\mathbb{P}\right)$, and the functions $V^N:\RR^d\to\RR_+$ and $F:\RR^d\times \RR_+ \to \RR^d$ are continuous and will be specified in the next paragraph.

 Thus we are assuming that the  system of $N$ particles is subject to random
dispersal, modelled as $\alpha$-stable L\'{e}vy processes. In this model randomness may be due to external sources, for instance unpredictable irregularities
of the environment (like obstacles, changeable soils, varying visibility). The moderate interaction is represented by the form of $V^N$ which writes as
\[ V^N(x) = N^\beta V(N^{\frac{\beta}{d}}x),\]
for some function $V:\RR^d\to\RR_+$ and $\beta \in (0,1)$. The case $\beta=0$ would correspond to the long range interaction, and the case $\beta=1$ to the nearest neighbor interaction. Further hypotheses on $V$ and $\beta$ are given in Assumption \ref{assump} below. 

We are interested in the  bulk behaviour of the whole population  of the  particles, 
and therefore a natural object for mathematical investigation is the 
probability measure-valued empirical process, defined as follows: let \[S_{t}^{N}:=\frac{1}{N}\sum_{i=1}^{N}\delta_{X_{t}^{i,N}}\] be the empirical 
process associated to  $\{X_{t}^{i,N}\}_{i=1,...,N}$, where
$\delta_{a}$ is the delta Dirac measure concentrated at $a$.
The drift term which appears in \eqref{itoass} describes the interaction of the $i$-th particle located at $X_t^{i,N}$
 with the random field $S_t^N$ generated by the whole system of particles at time $t$.The usual procedure includes the following steps: 
\begin{itemize}
\item show the convergence of the process $S^N$ to a deterministic measure process $S^\infty$;
\item and then, identify $S^\infty$ as the weak solution to some suitable PDE.
\end{itemize}

  In our case the
dynamics of the empirical measure is fully determined by  It\^{o}'s formula: when we apply it to $\phi\big(  X_{t}^{i,N}\big)$, for any test function $\phi\in C_{0}^{\infty}(\mathbb{R}^{d})$ which is smooth and compactly supported, we obtain that the empirical measure $S_{t}^{N}$ satisfies%
\begin{align}
\left\langle S_{t}^{N},\phi\right\rangle &=\left\langle S_{0}^{N}%
,\phi\right\rangle +\int_{0}^{t}\Big\langle S_{s}^{N},F\big(\cdot \;,\;(V^{N}\ast
S_{s}^{N})\big)\;\nabla\phi\Big\rangle\; ds \notag \\
&+ \frac{1}{2}\int_{0}^{t}\left\langle S_{s}^{N}, \mathcal{L} \phi\right\rangle
ds+   \frac{1}{N}\sum_{i=1}^{N}\int_{0}^{t} \int_{\RR^{d}-\{0\}} \big(\phi(X_{s_{-}}^{i,N}+z) - \phi(X_{s_{-}}^{i,N})\big)\;   d{\mathcal N}^{i}(dsdz)
\label{ident fro S^N},
\end{align}
where $\mathcal L$ is the non-local operator corresponding to a symmetric $\alpha$-stable L\'evy process,  defined as
\begin{equation}\label{OpeL}
\mathcal{L}\phi (x) =  \int_{\RR^{d}-\left\{  0 \right\}}  \big(\phi(x+ z)-\phi(x) -{\bf 1}_{\{|z|\leq 1\}} \nabla\phi(x)\cdot z\big) \;  d\nu(z),
\end{equation}
and $(\nu,{\mathcal N})$ are characteristics coming from the L\'evy-It\^o decomposition Theorem, which satisfy some conditions related to $\alpha \in (1,2)$, see Section \ref{ssec:levy} below for more details. As a example which satisfies the conditions of our main result, let us give the  fractional Laplacian $-(-\Delta)^{\frac{\alpha}{2}}$,  for some $\alpha \in (1,2)$, which corresponds to \eqref{OpeL} with 
\[ d\nu(z) = \frac{K_\alpha \; dz}{|z|^{1+\alpha}}.\]
Our interest lies in the investigation of the behaviour of the dynamics of the
processes $t\mapsto S_{t}^{N}$ in the limit  $N\rightarrow\infty$.  This kind of  problems
was considered  by Oelschl\"{a}ger \cite{Oel1},  Jourdain and  M\'el\'eard  \cite{Mela},
 M\'el\'eard and Roelly-Coppoletta \cite{MelRoelly},
when the sytem of SDE is driven by standard Brownian motions (see also the introduction in \cite{flan}). Up to our knowledge, the case of L\'evy processes driven dynamics has not been fully investigated in the literature. Let us now state the main result of this paper.

\subsection{Assumptions and main result} \label{ssec:mainresult}
 
%
% and it is endowed with the norm\[ \|\phi\|_{\HH^{\epsilon}}:=\bigg(\sum_{|\alpha | \leq 2} \big\|D^\alpha \phi\big\|_{\LL^\epsilon}^\epsilon\bigg)^{1/\epsilon}.\] 
%*** THIS IS NOT CORRECT?? **** Should be
%\[ \|\phi\|_{\HH^{\epsilon}}:=\bigg(\sum_{|\alpha | \leq \epsilon} \big\|D^\alpha \phi\big\|_{\LL^2}^2\bigg)^{1/2} \ ???.\]  }

In the following we denote by $\|\cdot\|_{\LL^p}$ the usual $\LL^p$-norm on $\RR^d$, and by $\|\cdot\|_{\LL^p\to \LL^q}$ the usual operator norm. Moreover, for any $\epsilon \in \RR$,  we denote by $\HH^\epsilon=\HH^\epsilon(\RR^d)$ the usual 
Bessel space of all
functions $u\in \LL^{2}(  \mathbb{R}^{d})  $ such that
\[
\left\Vert u\right\Vert _{\HH^{\epsilon}}^{2}:=\Big\Vert \mathcal{F}^{-1}\big[(1+|\lambda|^{2}\big)^{\frac{\epsilon}{2}%
} \mathcal{F}(u)(\lambda)\big]\Big\Vert _{\LL^{2}  }
^{2}<\infty
\]
where $\mathcal{F}$ denotes the Fourier transform of $u$.

For every $\epsilon \in\mathbb{R}$, the Sobolev spaces $\mathbb{W}^{\epsilon ,2}(\mathbb{R}^{d})  $ are well defined, see \cite{Triebel} for the material
needed here. For positive $\epsilon $ the restriction of $f\in \mathbb{W}^{\epsilon 
,2}(  \mathbb{R}^{d})  $ to a ball $\mathcal{B}(  0,R) \subset \mathbb{R}^d  $ is in
$\mathbb{W}^{\epsilon ,2}(  \mathcal{B}(  0,R)  )  $.   The Sobolev spaces $\HH^{\epsilon}$  and $\mathbb{W}^{\epsilon ,2}(\mathbb{R}^{d})  $   have equivalent norms. Let us list below our assumptions which we use to derive the macroscopic limit:
\begin{assumption} Take $\alpha \in (1,2)$. \label{assump}
We assume that there exists a  continuous probability density  $V:\RR^d\to \RR_+$   such that
\begin{itemize}
\item $V$ is compactly supported and symmetric,
\item  $V^{N}(x)=N^{\beta}V(N^{\frac{\beta}{d}}x)$ for some $\beta >0$,
\item $V \in \HH^{\epsilon + \delta}(\RR^d)$, for some $\varepsilon,\delta >0$,
\end{itemize}
There are further conditions on $(\beta,\varepsilon,\delta)$. For a technical reason that we will appear later (more precisely in Lemma \ref{lemma martingale 1}), we need to assume that $\epsilon$ satisfies 
\begin{equation} \frac d 2 < \epsilon < \frac{(1-\beta)d}{2\beta}-\Big(1-\frac{\alpha}{2}\Big) \label{eq:main}\end{equation}
and that $\delta$ satisfies
\begin{equation}
\label{eq:delta} 
1-\frac{\alpha}{2} < \delta \leq \frac{(1-\beta)d}{2\beta}-\epsilon.
\end{equation}
This is our main technical assumption. 
The inequality which appears in \eqref{eq:main} gives an extra condition on $\beta$ which reads as
\[ 0 < \beta < \frac{1}{2+\frac{2-\alpha}{d}}.\] 
Let us note that the sequence $\{V^N\}_N$ is a family of \emph{mollifiers} which will allow us to introduce a mollified version of the empirical density (see below).  Moreover we assume: 
\begin{itemize}
\item $F\in \mathbb L^{\infty}( \mathbb{R}^{d}\times\mathbb{R})  \cap
{\rm Lip}(\mathbb{R}^{d}\times\mathbb{R}).$
\item The sequence of measures $\{S_0^N\}_N$  converges weakly to $u_0(\cdot)dx$ in probability, as $N\to\infty$, where $u_0 \in \mathbb{L}^1(\mathbb{R}^d)$.
\item The sequence of mollified initial measures $\{V^N \ast S_0^N\}_N$ is uniformly bounded in the Sobolev space $\mathbb{H}^\epsilon$, namely: 
\[
\sup_{N \in \mathbb{N}}\big\|  V^{N}\ast S_0^{N}  \big\|_{\HH^{\epsilon}}  <\infty.\]
\end{itemize}
\end{assumption}

Let us introduce the \emph{mollified empirical measure} (the theoretical analogue of
the numerical method of kernel smoothing) $g_{t}^{N} $ defined as%
\[
g_{t}^{N}(x)=\left(  V^{N}\ast S_{t}^{N}\right)  \left(  x\right), \qquad x \in \RR^d .
\]
We are now ready to state our main result:

\begin{theorem}
\label{Thm 1}  We assume Assumption \ref{assump}. Then, for every
$\eta \in( \frac d 2,\epsilon)$, the sequence of processes $\{(g_t^{N})_{t\in[0,T]}\}_N$ converges in probability with respect to the 
\begin{itemize}
\item \emph{weak star topology of $\LL^{\infty}\left(  [0,T]\; ;\;\LL^{2}(
\mathbb{R}^{d})\right)   $, }

\item \emph{weak topology of $\LL^{2}\left(  [0,T]\; ;\; \HH^{\epsilon}(
\mathbb{R}^{d})  \right)  $ }

\item \emph{strong topology of $\LL^{2}\left( [ 0,T]\; ;\; \HH_{\rm loc}^{\eta}(
\mathbb{R}^{d})  \right)  $ }
\end{itemize}
as $N\rightarrow\infty$, to the unique weak solution of the  non-local PDE
\begin{equation}
\partial_{t}u(t,x)+\mathrm{div}(F(x,u)u)- \mathcal{L} u(t,x)=0\,, \qquad
u|_{t=0}=u_{0}, \label{PDEintr}%
\end{equation}
where $ \mathcal{L}$ has been defined in \eqref{OpeL} and is the operator of a symmetric $\alpha$-stable L\'evy process. Namely, for all $\phi\in C_{0}^{\infty}(\mathbb{R}^{d})$ it holds
\begin{equation}\label{limiteq}
\left\langle u(t,\cdot),\phi\right\rangle =\left\langle u_{0},\phi\right\rangle
+\int_{0}^{t}\left\langle u,F(\cdot,u)\nabla\phi\right\rangle ds+\frac{1}{2}%
\int_{0}^{t}\left\langle u,\mathcal{L}\phi\right\rangle ds.
\end{equation}

\end{theorem}

The literature concerning the type of equations mentioned above is immense. We will only give a partial and incomplete 
survey of some parts that we feel more relevant for this paper.  For a more complete discussion and many more
references, we refer the reader to the nice works \cite{ali,Al,Andre,Dro,Vaz}. 
 A large variety of phenomena in physics and finance are modelled by linear anomalous diffusion equations, see \cite{Vaz}. Fractional conservation laws are generalizations of convection-diffusion equations
and  appear in some physical models for over-driven detonation in gases \cite{Cla} and semiconductor growth \cite{Woy}, and in areas like dislocation dynamics, hydrodynamics, and molecular biology.

\medskip

About the propagation of chaos phenomena, let us also mention that it has recently
been studied in the context well beyond that of the Brownian motion, namely, in
the situation where the driving Brownian motions have been replaced by L\'{e}vy
processes and anomalous diffusions. We  mention  the works    \cite{Biler,Mela2,Mela3}. In
\cite{Mela2} the authors consider a  singular fractal conservation  and they construct a McKean-style non-linear process
and then use it to develop an interacting particle system whose empirical measure  strongly converges
to the solution.   In  \cite{Biler}  a weak result of this type has been obtained.  In 
\cite{Mela3} the authors   deal with an  interacting particle system whose empirical measure  strongly converges
to the solution of a one-dimensional fractional non-local conservation law via the non-linear martingale problem associated to the PDE.

\medskip

The main result of this paper is to generalize the 
propagation of chaos Theorem given by Oelschl\"{a}ger \cite{Oel1}  for systems of stochastic differential equations driven by  L\'{e}vy noise, which include non-linear terms as
\[
\int_{0}^{t}\Big\langle S_{s}^{N},F\big(\cdot\; ,\; (V^{N}\ast
S_{s}^{N})\big)\; \nabla\phi\Big\rangle ds
\]
where $\phi$ is a smooth test function. Since $S_{t}^{N}$ converges only
weakly, it is required that $V_{N}\ast S_{t}^{N}$ \textit{converges
uniformly}, in the space variable, in order to pass to the limit. Maybe in
special cases one can perform special tricks but the question of uniform
convergence is a natural one in this problem and it is also of independent
interest, hence we investigate when it holds true. Notice that the moderate interaction assumption in \cite{Oel1} reads as $\beta \in (0,\frac{d}{d+2})$  whereas here we obtain $\beta \in (0,\frac{d}{2-\alpha + 2d})$, where $\alpha$ is one of the main characteristics of our L\'evy process. 
The case $\beta=1$ is much more challenging, and up to our knowledge, not solved for the time being.

Finally, we mention that 
 our source of inspiration was the paper \cite{flan}  where the authors use a semi-group approach in order to study the propagation of chaos for a system of Brownian particles with proliferation, and there they obtain the condition $\beta \in (0,\frac12)$, which already improved the one of \cite{Oel1}. The differences mainly rely on martingale estimates which, in the present case, are of L\'evy type, and therefore contain jumps. We keep the semi-group approach exposed in \cite{flan}, but all our main technical lemmas involve new analytic tools. 

\medskip

Here follows an outline of the paper: in Section \ref{sec:preli} below we gather some well-known results that we use in the paper (with precise references for all the proofs) concerning stable L\'evy processes, semi-group properties and criteria of convergence.  In Section \ref{sec:result} we prove Theorem \ref{Thm 1},  by following three main steps: first, obtain uniform bounds for the mollified empirical measure; second, find compact embeddings to extract convergent subsequences; and third, pass to the limit and use a uniqueness result  for the solution to \eqref{PDEintr}.

\section{ Preliminaries}
\label{sec:preli}

\subsection{ Stable L\'{e}vy processes} \label{ssec:levy}

We list a collection of  definitions and  classical results that can be found in any textbook or monography on L\'{e}vy processes. We refer to \textit{Applebaum (2009)} \cite{App}, \textit{Kunita (2004)} \cite{Kunita} and \textit{Sato (2013)} \cite{Sato} where all the results and definitions presented in this section are treated.

\begin{definition}[\textbf{L\'{e}vy process}] \label{appendix: def levy process} \textit{A process $L=(L_t)_{t \geq 0}$ with values in $\RR^d$ defined on a probability space $(\Omega, \fF, \PP)$ is a \emph{L\'{e}vy process} if the following conditions are fulfilled:
\begin{itemize}
\item[1.] $L$ starts at 0 $\PP-$a.s., i.e. $\PP(L_0=0)=1;$ 
\item[2.] $L$ has independent increments, i.e. for $k \in \NN$ and $0 \leq t_0 < \dots < t_k$, 
\begin{align*}
L_{t_1} - L_{t_0},  \dots, L_{t_k} - L_{t_{k-1}} \quad \text{are independent };
\end{align*}
\item[3.] $L$ has stationary increments, i.e., for $0 \leq s \leq t$, $L_{t} - L_s $ is equal in distribution to $ L_{t-s}\; ;$
\item[4.] $L$ is stochastically continuous, i.e. for all $t \geq 0$ and $\varepsilon>0$
\begin{align*}
\displaystyle \lim_{s \rightarrow t} \PP (|L_t - L_s| > \varepsilon) =0.
\end{align*}
\end{itemize}}
\end{definition}

The reader can find the proof of the following result in \cite[Theorem 2.1.8]{App}:

\begin{proposition} Every L\'{e}vy process has a c\`{a}d-l\`{a}g modification that is itself a L\'{e}vy process.
\end{proposition}
Due to this fact, we assume moreover that every L\'{e}vy process has  almost surely c\`{a}d-l\`{a}g paths. Following \cite[Section 2.4]{App} and  \cite[Chapter 4]{Sato} we state the \textit{L\'{e}vy-It\^{o} decomposition theorem} which characterizes  the paths of a  L\'{e}vy process in the following way.

\begin{theorem}[\textbf{L\'{e}vy-It\^{o} decomposition Theorem}] \label{appendix: thm levy ito decomposition} 

Consider $b \in \RR^d$, $\sigma$ a positive definite matrix of $\RR^{d \times d}$ and $\nu$ a measure defined on the Borelians of $\RR^d$ satisfying $\nu(\{ 0\})=0$ and $\int_{\RR^d} (1\wedge |z|^2)\nu(dz) < \infty$. 
\medskip

Then there exists a probability space $(\Omega, \fF, \PP)$ on which four independent L\'{e}vy processes exist, $L^1, L^2, L^3$ and $L^4$ with the following properties: \begin{itemize}\item $L^1_t= bt$, for all $t \geq 0$ is called a constant drift ; \item $L^2$ is a Brownian motion with covariance $\sqrt{\sigma}$ ;  \item $L^3$ is a compound Poisson process ; \item $L^4$ is a square integrable (pure jump) martingale with an a.s.~countable number of jumps of magnitude less than 1 on every finite time interval. \end{itemize}
 Hence, for $L=L^1 + L^2 + L^3+ L^4$ there exists a probability space on which  $(L_t)_{t \geq 0}$ is a L\'{e}vy process such that $\mathbb{E}\big[e^{i<\xi,L_t>}\big] =\exp(-t\psi(\xi))$ where the characteristic exponent $\psi(\xi)$ is given by 
\[
\psi(\xi)= i \langle b,\xi \rangle - \frac{1}{2} \langle\sigma \xi, \xi\rangle + \int_{\RR^d} \big(e^{i \langle \xi,z\rangle} -1 - i \langle \xi,z \rangle {\bf 1}_{\{ |z| < 1 \}} \big) \nu(dz), \quad \xi \in \RR^d.
\] 
Conversely, given a L\'{e}vy process defined on a probability space, there exists $b \in \RR^d$, a Wiener process $(B_t)_{t \geq 0}$, a covariance matrix $ \sqrt{\sigma} \in \RR^{d \times d}$ and an independent Poisson random measure $\mathcal{N}$ defined on $\RR^+ \times (\RR^d - \{ 0\})$ with intensity measure $\nu$ such that, for all $t \geq 0$,
\begin{equation}\label{appendix: eq levy ito decomposition}
L_t = bt + \sqrt{\sigma} B_t + \int_0^t  \int_{0 < |z| < 1}  z\;  \tilde{\mathcal{N}}(dsdz) + \int_0^t \int_{|z|>1} z\;  \mathcal{N}(dsdz). 
\end{equation}
More precisely, the  Poisson  random measure $\mathcal{N}$ is defined by
\[
\mathcal{N}((0,t]\times U)=\sum_{s\in (0,t]} {\bf 1}_{U} (L_s-L_{s_{-}}) \quad \text{ for any } U \in \mathcal{B}(\RR^{d}-\{0\}), \; t>0,
\]
and the compensated Poisson  random measure is given by  \[\tilde{\mathcal N}((0,t]\times U)= \mathcal{N}((0,t]\times U)- t \nu(U).\] 
\end{theorem}

%With the notation of the last theorem, for all $t \geq 0$, we have
%\begin{align*}
%L^1_t &= bt, \\
%L^2_t &=  \sqrt{\sigma} B_t, \\
%L^3_t &=  \int_0^t \int_{|z|>1} z N(ds,dz), \\
%L^4_t &= \int_0^t  \int_{0 < |z| < 1}  z \tilde N(dsdz).
%\end{align*}
%The associated Poisson  random measure is defined by
%\[
%N((0,t]\times U)=\sum_{s\in (0,t]} {\bf 1}_{U} (L_s-L_{s_{-}}) \ U \in \mathcal{B}(\RR^{d}-{0}), \quad t>0.
%\]
%The compensated Poisson  random measure is given by  $\tilde{N}((0,t]\times U)= \textit{N}((0,t]\times U)- t {\color{red}\nu(U)}$.

Throughout this paper we consider furthermore that  $(L_t)_{t\geq 0}$ is a symmetric $\alpha$-stable process for some $\alpha \in (1,2)$. We recall some facts about \emph{symmetric $\alpha$-stable processes}. These can be completely defined via their characteristic function, which is given by (see \cite{Sato} for instance)
\[
\EE\big[ e^{i<\xi, L_t>}\big]= e^{-t\psi(\xi)},
\]
where 
\begin{equation}\label{eq:psi}
\psi(\xi)= \int_{\RR^{d}} \big(1-e^{i\langle\xi,z\rangle}+ i\langle\xi,z\rangle {\bf 1}_{\{|z|\leq1\}}\big) \; d\nu(z),
\end{equation}
and the L\'{e}vy measure $\nu$ with $\nu(\left\{ 0  \right\})=0$ is given by 
\[
\nu(U)=\int_{S^{d-1}} \int_{0}^{\infty} \frac{{\bf 1}_{U}(r\theta)}{r^{\alpha+d}} \; dr \; d\mu(\theta) , 
\]
where $\mu$ is some symmetric finite measure concentrated on the unit sphere $S^{d-1}$, called \textit{spectral measure}
of the stable process $L_{t}$.  

Moreover, we assume the following additional property: for some constant $C_{\alpha}> 0$, 
\begin{equation} 
\psi(\xi)\geq C_{\alpha} |\xi|^{\alpha}, \qquad \text{ for any }\; \xi\in \RR^{d}.
\end{equation}
We remark that the above condition is equivalent to the fact that the support of the spectral measure $\mu$ is not contained in the 
proper linear subspace of $\RR^{d}$, see \cite{Priola}.
 Finally, note that the
L\'{e}vy-It\^{o} decomposition now reads as
\begin{equation}\label{levy}
L_t= \int_0^t \int_{|z|>1} z\; \mathcal{N}(dsdz) + \int_0^t  \int_{0 < |z| < 1}  z\; \tilde{\mathcal{N}}(dsdz),
\end{equation}
which corresponds to \eqref{appendix: eq levy ito decomposition} with $b=0$ and $\sigma=0$. 
Now, we recall the following well-known properties about the symmetric $\alpha$-stable processes (see for instance \cite[Proposition 2.5]{Sato} and 
\cite[Section 3]{Priola}).

\begin{proposition} Let $\mu_t$ be the law of the symmetric $\alpha$-stable process $L_t$. Then

\begin{enumerate}

\item  \emph{(Scaling property)}.    For any $\lambda>0$,  $L_t$ and $\lambda^{-\frac{1}{\alpha}}L_{\lambda t}$ have the same finite dimensional law. 
 In particular, for any $t>0$ and $A\in \mathcal{B}(\RR^{d})$,  $\mu_t(A)=\mu_1(t^{-\frac{1}{\alpha}}A)$.

\item \emph{(Existence of smooth density)}.  For any $t>0$, $\mu_t$ has a smooth density $\rho_t$ with respect to the Lebesgue measure,  which is given by
\[
\rho_t(x)=\frac{1}{(2\pi)^{d}}  \int e^{-i\langle x,\xi\rangle} \ e^{-t\psi(\xi)}   \  d\xi.
\]
Moreover $\rho_{t}(x)=\rho_{t}(-x)$ and for any $k\in \mathbb{N}$,  $\nabla^k \rho_t \in \LL^{1}(\RR^{d})$. 
\item  \emph{(Moments)}. For any $t> 0$, if $\beta< \alpha$, then $\mathbb{E}\big[|L_t|^{\beta}\big]< \infty$, and if $\beta\geq \alpha$ then
$\mathbb{E}\big[|L_t|^{\beta}\big]=\infty$.

\end{enumerate}

\end{proposition}

\subsection{Semi-group and Sobolev spaces\label{subsect Bessel}}
From now on, we denote $\mathbb{L}^p=\mathbb{L}^p(\RR^d)$ and $\HH^\epsilon = \HH^\epsilon(\RR^d)$.  The family of operators,
for $t\geq0$,%
\[
\left(  e^{t\mathcal{L}}f\right)  \left(  x\right)  =\int_{\mathbb{R}^{d}} p_t(y-x)  f(y) \;   dy
\]
defines a Markov semi-group in each space $ \HH^{\epsilon}$ ; with little abuse of notation, we write $e^{t\mathcal{L}}$
for each value of $\epsilon$.

We consider  the operator $A:D(  A)  \subset
\LL^{2}\rightarrow \LL^{2}  $ defined as $Af=\Delta f$. For $\epsilon\geq 0 $,  the  fractional powers $(
I-A)^{\epsilon}$ are well defined for every $\epsilon\in\mathbb{R}$ and
$\Vert (  I-A)^{\epsilon/2}f\Vert _{\LL^2}$ is equivalent to the norm of $\HH^{\epsilon}  $. We recall also 

\begin{proposition} \label{prop} For every
$\epsilon\geq 0$ and $\alpha \in (1,2)$, and given $T>0$, there is a constant $C_{\epsilon,\alpha,T}$ such that, 
for any $t\in(0,T]$, 
\begin{equation}\label{s1}
\left\Vert \left(  I-A\right)^{\epsilon}e^{t\mathcal{L}}\right\Vert _{\LL^{2}\rightarrow
\LL^{2}}\leq\frac{C_{\epsilon,\alpha, T}}{t^{2 \epsilon/\alpha}}.%
\end{equation}
\end{proposition}

\begin{proof}We only sketch the proof, which is standard.

\medskip

\noindent \textsc{Step 1}: Using the scaling property $\rho_t(x)=t^{-\frac{d}{\alpha}}  \rho_{1}(t^{-\frac{1}{\alpha}}x)$, we arrive at
\[
\big|\nabla^{k+m} \rho_t \ast f\big|(x)\leq \frac{t^{-\frac{k}{\alpha}}}{t^{\frac{d}{\alpha}}} 
\int_{\RR^{d}} \ \big|\nabla^{m}f(z)\big|\  \big|(\nabla^{k}\rho_1)(t^{-\frac{1}{\alpha}}z-t^{-\frac{1}{\alpha}}x)\big| \ dz.
\]
Thus
\[
\big\|\nabla^{k+m} \rho_t \ast f \big\|_{\LL^2}   \leq t^{-\frac{k}{\alpha}} \big\|\nabla^{m}f \big\|_{\LL^2}  \   \big\|\nabla^{k} \rho_1 \big\|_{\LL^1}
\]
It follows that  
\[
\big\|\nabla^{k} e^{t\mathcal{L}} \big\|_{\HH^{m+k}  \rightarrow \HH^{m}}   \leq  C t^{-\frac{k}{\alpha}}.
 \]

\medskip

\noindent \textsc{Step 2}: The sub-Markovian   of $e^{t\mathcal{L}}$ (see \cite{App})  implies that
\[
\|e^{t\mathcal{L}} f\|_{\LL^2}\leq \| f\|_{\LL^2}. 
\]

\medskip

\noindent \textsc{Step 3}: Finally,  by a standard interpolation inequality  we conclude the proposition for any $\epsilon$.
\end{proof}

\subsection{Positive operator}

\begin{lemma}\label{positi} We assume that $f\in \mathbb{L}^{1} \cap \mathbb{L}^{2}$ and $f\geq 0$. Then for any $t\geq 0$,  $(  I-A)  ^{\epsilon/2}\; e^{t\mathcal{L}}f\geq0$.\end{lemma}

\begin{proof} We observe that 
$g:=\left(  I-A\right)^{\epsilon/2}e^{t\mathcal{L}}f\in \mathbb{L}^1$. In fact, 
\[
\| g\|_{\mathbb{L}^{1}}= \big\| \left(  I-A\right)^{\epsilon/2}e^{t\mathcal{L}} f    \big\|_{\mathbb{L}^{1}} 
\leq \| f\|_{\mathbb{L}^{1}}  \ \    \|\left(  I-A\right)^{\epsilon/2}e^{t\mathcal{L}}    \big\|_{\mathbb{L}^{1}} \ < \ \infty.
\]
    In order to prove that the function $g:=\left(  I-A\right)
^{\epsilon/2}e^{t\mathcal{L}}f$ is non-negative,  by  Bochner's  Theorem  it is sufficient to prove that its
Fourier transform $\widehat{g}$ is definite positive, namely
$\operatorname{Re}\big[\sum_{i,j=1}^{n}\widehat{g}\left(  \lambda_{i}-\lambda
_{j}\right)  \xi_{i}\overline{\xi}_{j}\big]\geq0$ for every $n\in\mathbb{N}$,
$\lambda_{i}\in\mathbb{R}^{d}$ and $\xi_{i}\in\mathbb{C}$, $i=1,...,n$. We
have
\[
\widehat{g}\left(  \lambda\right)  =\big(  1+\left\vert \lambda\right\vert
^{2}\big)  ^{\epsilon/2}e^{-t \psi(\lambda)}
\widehat{f}\left(  \lambda\right)
\]
where $\psi$ has been defined in \eqref{eq:psi}. Thus we have to prove that, given $n\in\mathbb{N}$, $\lambda_{i}%
\in\mathbb{R}^{d}$ and $\xi_{i}\in\mathbb{C}$, $i=1,...,n$, one has
\[
\operatorname{Re}\bigg[\sum_{i,j=1}^{n}\big(  1+\left\vert \lambda_{i}-\lambda
_{j}\right\vert ^{2}\big)  ^{\epsilon/2}e^{-t   \psi(\lambda_{i}
-\lambda_{j})}    \widehat{f}\left(  \lambda_{i}-\lambda
_{j}\right)  \xi_{i}\overline{\xi}_{j}\bigg]\geq0
\]
namely
\begin{multline*}
\sum_{i=1}^{n}\operatorname{Re} \big[  e^{-t \psi(0)}\widehat{f}\left(  0\right)  \xi_{i}%
\overline{\xi}_{i}\big]\\+\sum_{i<j}\big(  1+\left\vert \lambda_{i}-\lambda
_{j}\right\vert ^{2}\big)  ^{\epsilon/2}  e^{-t\psi(\lambda_{i}
-\lambda_{j})}      \left(  \operatorname{Re}\big[\widehat{f}\left(
\lambda_{i}-\lambda_{j}\right)  \xi_{i}\overline{\xi}_{j}\big]+\operatorname{Re}\big[
\widehat{f}\left(  \lambda_{j}-\lambda_{i}\right)  \xi_{j}\overline{\xi}%
_{i}\big]\right)  \geq 0.
\end{multline*}
We observe 
\[
\sum_{i=1}^{n} \operatorname{Re} \big[  e^{-t \psi(0)} \widehat{f}\left(  0\right)  \xi_{i}
\overline{\xi}_{i}\big] = 
|\xi|^{2} e^{-t \psi(0)} \operatorname{Re} \big[\widehat{f}(0)\big]=
|\xi|^{2} e^{-t \psi(0)} \int f(x) \ dx \geq 0. 
\]
 Since  $f$ is non-negative,   for $i\neq j $ we obtain 
\[
\operatorname{Re}\big[\widehat{f}\left(  \lambda_{i}-\lambda_{j}\right)  \xi
_{i}\overline{\xi}_{j}\big]+\operatorname{Re}\big[\widehat{f}\left(  \lambda_{j}%
-\lambda_{i}\right)  \xi_{j}\overline{\xi}_{i}\big]\geq0.
\]
Using these two facts above we get the result.
\end{proof}

\subsection{Maximal function}
\label{sec:max} 
Let $f$ be a locally integrable function on $\RR^{d}$.  The \emph{Hardy-Littlewood maximal function} is
defined by
\[
\mathbb{M}f(x)=\sup_{0< r< \infty} \bigg\{\frac{1}{|\mathcal{B}_r|} \int_{\mathcal{B}_r} f(x+y) \ dy\bigg\},
\]
where $\mathcal{B}_r=\left\{ x\in\RR^{d} \ : |x|< r   \right\}$. The following results can be found in \cite{Stein}.

\begin{lemma} \label{maxi} For all $f\in \mathbb{W}^{1,1}(\RR^d)$ there exists a constant $C_d>0$ and a Lebesgue zero set $E\subset \RR^d$ such that
\[
|f(x)-f(y)|\leq C_d \; |x-y|\; \Big(\mathbb{M}|\nabla f|(x)   +  \mathbb{M}|\nabla f|(y)\Big) \qquad \text{for any } x,y \in \RR^d\backslash E. 
\]
Moreover, for all $p>1$   there exists a constant  $C_{d,p}>0$ such that for all  $f\in \LL^{p}(\RR^{d})$
\[
 \|\mathbb{M} f\|_{\LL^{p}} \leq C_{d,p}  \;   \| f\|_{\LL^{p}}.
\]
\end{lemma}

\subsection{Criterion of convergence in probability}

\begin{lemma}[Gyongy-Krylov \cite{Gyon}] \label{GK}

Let $\{X_{n}\}_{n\in\NN}$ be a sequence of random elements in a Polish
space $\Psi$ equipped with the Borel $\sigma$-algebra. Then $X_{n}$ converges in
probability to a $\Psi$-valued random element if, and only if, for each pair
$\{X_{\ell},X_{m}\}$ of subsequences, there exists a subsequence $\{v_{k}\}$,
\[
v_{k}= (X_{\ell(k)}, X_{m(k)}),
\]
converging weakly to a random element $v$ supported on the diagonal set
\[
\big\{ (x,y) \in\Psi\times\Psi: x= y \big\}.
\]

\end{lemma}

\section{ Proof of Theorem \ref{Thm 1}} \label{sec:result}

The strategy is as follows:
\begin{enumerate}
\item Recall that we already defined
\begin{equation}
g_{t}^{N}(x)=(V^{N}\ast S_{t}^{N})(x)=\int_{\mathbb{R}^{d}}V^{N}(x-y)  dS_{t}^{N}\left(  y\right).  \label{def g}%
\end{equation}
Using a mild (semi-group) formulation of the identity satisfied by
$g_{t}^{N}$ (Section \ref{sec:eq}), we  prove  uniform bounds in Section \ref{sec:bound} (Lemma \ref{estimation}, Lemma \ref{lemma martingale 1} and Lemma \ref{estimation2}).

\item Then we apply compactness arguments and Sobolev embeddings to have
subsequences which converge so as to pass to the limit (Sections \ref{subsect compactness} and \ref{sec:limit}). 

\item Finally, we use previous works to obtain that the weak solution to \eqref{PDEintr} is unique in our class of convergence (Section \ref{sec:unique}). 
\end{enumerate}

\subsection{The equation for $g_{t}^{N}$ in
mild form} \label{sec:eq}

 We want to
deduce an identity for $g_{t}^{N}(x)$ from \eqref{ident fro S^N}. 
 For $h>0$, let us consider the regularized
function $e^{h\mathcal{L}} V^{N} $. For every given
$x\in\mathbb{R}^{d}$ let us take, in identity \eqref{ident fro S^N}, the test
function $\phi_{x}\left(  y\right)  =\left(  e^{h\mathcal{L}}
V^{N} \right)  \left(  x-y\right)  $. We get 
\begin{align*}
\big(  e^{h\mathcal{L}}&g_{t}^{N}\big)  (x)    =\big(  e^{h\mathcal{L}}g_{0}^{N}\big)  (x)+\int_{0}^{t}\left\langle S_{s}%
^{N},F(\cdot \; ,\; g_{s}^{N})\nabla\big(  e^{h\mathcal{L}} V^{N}
\big)  \left(  x-\cdot\right)  \right\rangle ds\\
&  +\frac{1}{2}\int_{0}^{t}\mathcal{L}\big(  e^{h\mathcal{L}}g_{s}%
^{N}\big)  (x)ds\\
& + \frac{1}{N}\sum_{i=1}^{N} \int_{0}^{t} \int_{\RR^{d}-\{0\}} \Big\{  \big(
e^{h\mathcal{L}} V^{N} \big)  \big(  x-X_{s_{-}}^{i,N}+ z  \big)  -   
   \big(e^{h\mathcal{L}} V^{N}\big)\big( x-X_{s_{-}}^{i,N}\big)  \Big\} d\tilde{\mathcal N}^{i}(dsdz).
\end{align*}
Let us write, in the sequel,%
\[
\left\langle S_{s}^{N},F(\cdot\; ,\; g_{s}^{N})\nabla\big(  e^{h\mathcal{L}}
V^{N} \big)  \left(  x-\cdot\right)  \right\rangle =:\left(
\nabla\big(  e^{h\mathcal{L}} V^{N} \big)  \ast\left(
F(\cdot\; , \; g_{s}^{N})S_{s}^{N}\right)  \right)  \left(  x\right).
\]
and similarly for similar expressions. Following a standard procedure, used for instance by \cite{DaPrZab}, we may
rewrite this equation in mild form:
\begin{align*}
&e^{h\mathcal{L}}g_{t}^{N}    =e^{t\mathcal{L}}\big(  e^{h\mathcal{L}}%
g_{0}^{N}\big)  +\int_{0}^{t}e^{\left(  t-s\right) \mathcal{L}%
}\left(  \nabla e^{h\mathcal{L}} V^{N} \ast\left(
F(\cdot\; ,\; g_{s}^{N})S_{s}^{N}\right)  \right)  ds\\
&  +    \frac{1}{N}\sum_{i=1}^{N} \int_{0}^{t}  e^{\left(  t-s\right)\mathcal{L}} \int_{\RR^{d}-\{0\}}  \Big\{ \big(
e^{h\mathcal{L}} V^{N} \big)  \big(  x-X_{s_{-}}^{i,N}+ z  \big) -   
   \big(e^{h\mathcal{L}} V^{N} \big)\big( x-X_{s_{-}}^{i,N}\big) \Big\} d\tilde{\mathcal N}^{i}(dsdz).
\end{align*}
By inspection of the convolution explicit formula for $e^{\left(  t-s\right)
\mathcal{L}}$, we see that $e^{\left(  t-s\right)  \mathcal{L}}\nabla f=\nabla e^{\left(
t-s\right) \mathcal{L} }f$, and we can also use the semi-group property, hence we may
also write
\begin{align}
e^{h\mathcal{L}}g_{t}^{N}&=e^{(t+h)\mathcal{L}}g_{0}^{N}+\int%
_{0}^{t}\nabla e^{\left(  t+h-s\right) \mathcal{L}}\left( 
V^{N} \ast\left(  F(\cdot \; ,\; g_{s}^{N})S_{s}^{N}\right)  \right)  ds \notag \\ \label{mildh}
 & +\frac{1}{N}\sum_{i=1}^{N} \int_{0}^{t} \int_{\RR^{d}-\{0\}}   \bigg\{\big(
e^{(t-s+ h)\mathcal{L}} V^{N} \big)  \big(  x-X_{s_{-}}^{i,N} + z  \big)   \notag \\  
 & \qquad \qquad \qquad \qquad \qquad \qquad \qquad  - \big(e^{(t-s+ h)\mathcal{L}} V^{N} \big)\big( x-X_{s_{-}}^{i,N}\big) \bigg\}
     d\tilde{\mathcal N}^{i}(dsdz).
\end{align}
This is the identity which we use below. We can also pass to the limit as
$h \rightarrow0$ and deduce%
\begin{align}
g_{t}^{N}&=e^{t\mathcal{L}}g_{0}^{N}+\int_{0}^{t}\nabla
e^{\left(  t-s\right) \mathcal{L}}\left(  V^{N} \ast\left(
F(\cdot \; ,\; g_{t}^{N})S_{s}^{N}\right)  \right)  ds \notag \\
&+ \frac{1}{N}\sum_{i=1}^{N} \int_{0}^{t} \int_{\RR^{d}-\{0\}}  \Big\{ \big(
e^{(t-s)\mathcal{L}} V^{N} \big)  \big(  x-X_{s_{-}}^{i,N} + z \big) \notag \\
&  \qquad \qquad \qquad \qquad \qquad \qquad \qquad  -   
   \big(e^{(t-s)\mathcal{L}} V^{N} \big)\big( x-X_{s_{-}}^{i,N}\big) \Big\} d\tilde{\mathcal N}^{i}(dsdz). \label{eq:mild2}
\end{align}
In what follows we denote by $M_t^N$ the martingale 
\begin{multline}
\label{eq:martingale}
  M_t^N:=\frac{1}{N}\sum_{i=1}^{N} \int_{0}^{t} \int_{\RR^{d}-\{0\}}  \Big\{ \big(
e^{(t-s+ h)\mathcal{L}} V^{N} \big)  \big(  x-X_{s_{-}}^{i,N} + z  \big)  \\-   
   \big(e^{(t-s+ h)\mathcal{L}} V^{N} \big)\big( x-X_{s_{-}}^{i,N}\big) \Big\} d\tilde{\mathcal N}^{i}(dsdz).
\end{multline}

\subsection{Uniform $\LL^2$ bounds}
\label{sec:bound}
\subsubsection{  First  estimate on $g^{N}$}

\begin{lemma}\label{estimation} For any $t\in[0,T]$ and $N\in\mathbb{N}$, it holds
\[
\mathbb{E} \Big[  \big\Vert \left(  I-A \right) ^{\epsilon/2}g_{t}%
^{N}\big\Vert _{\LL^{2}}\Big]  \leq
C_{\epsilon}.
\]
\end{lemma}

\begin{proof}
\textsc{Step 1}. Let us use the $\LL^{2}$-norm $\left\Vert \cdot\right\Vert
_{\LL^{2}(  \Omega\times\mathbb{R}^{d})  }$ in the product space
$\Omega\times\mathbb{R}^{d}$ with respect to the product measure. From
\eqref{mildh} we  obtain 
\begin{align}
\big\Vert \left(  I-A\right)  ^{\epsilon/2}e^{h\mathcal{L}} g
_{t}^{N}&\big\Vert _{\LL^{2}(  \Omega\times\mathbb{R}^{d})  }\notag \\ &
\leq\big\Vert \left(  I-A\right)  ^{\epsilon/2}e^{(t+h)\mathcal{L}%
}g_{0}^{N}\big\Vert _{\LL^{2}(  \Omega\times\mathbb{R}%
^{d})  } \label{eq:1}\\
&  +\int_{0}^{t}\big\Vert \left(  I-A\right)  ^{\epsilon/2}\nabla
e^{\left(  t+h-s\right)\mathcal{L}}\left(   V^{N} \ast\left(
F(\cdot\; ,\; g_{s}^{N})S_{s}^{N}\right)  \right)  \big\Vert _{\LL^{2}(
\Omega\times\mathbb{R}^{d})  }ds \label{eq:2}\\
&  + \big\Vert  \left(  I-A\right)^{\epsilon/2} M_t^{N}    \big\Vert _{\LL^{2}(
\Omega\times\mathbb{R}^{d})  } \label{eq:3}
\end{align}
where $M_t^{N}$ has been defined in \eqref{eq:martingale}.

\medskip
\noindent \textsc{Step 2}. The first term \eqref{eq:1} can be estimated by%
\[
\big\Vert \left(  I-A\right)  ^{\epsilon/2}e^{(t+h)\mathcal{L}}g%
_{0}^{N}\big\Vert _{\LL^{2}(  \Omega\times\mathbb{R}^{d}) }  
\leq  \big\|  (  I-A)^{\epsilon/2} g_{0}^{N} \big\|_{\LL^{2}(\Omega\times\mathbb{R}^{d})} \leq C 
\]
where from now on we denote generically by $C>0$ any constant independent of
$N$. The boundedness of $\|  (  I-A)^{\epsilon/2} g_{0}^{N} \|_{\LL^{2}\left(\Omega\times\mathbb{R}^{d}\right)} $ is assumed  in  Assumption \ref{assump}. 

\medskip
\noindent \textsc{Step 3}. Let us come to the second term \eqref{eq:2} above:
%We have%
%{\color{red}
%\begin{align*}
%e^{\left(  t+h-s\right)  L}\left(  \nabla V^{N}\ast\left(  F(\cdot\; ,\; g_{t}%
%^{N})S_{s}^{N}\right)  \right)   &  =\sum_{j=1}^{d}e^{\left(  t+h-s\right)
%L}\partial_{j}\left(  V_{N}\ast\left(  F_{j}(\cdot\; ,\;g_{t}^{N})S_{s}^{N}\right)
%\right) \\
%&  =\sum_{j=1}^{d}\partial_{j}e^{\left(  t+h-s\right)  L}\left(  V^{N}%
%\ast\left(  F_{j}(\cdot\; ,\;g_{t}^{N})S_{s}^{N}\right)  \right) \\
%&  =:\nabla e^{\left(  t+h-s\right)  L}\left(  V^{N}\ast\left(
%F(\cdot\; ,\;g_{t}^{N})S_{s}^{N}\right)  \right)
%\end{align*}
%using the convolution formulation of the action of $e^{\left(  t+h-s\right)L}$ 
%
%**** this is useless, right? this is already written in the previous paragraph... ****}. Hence
\begin{align*}
&  \int_{0}^{t}\big\Vert \left(  I-A\right)  ^{\epsilon/2}\nabla
e^{\left(  t+h-s\right) \mathcal{L}}\left(  V^{N} \ast\left(
F(\cdot\; ,\;g_{s}^{N})S_{s}^{N}\right)  \right)  \big\Vert _{\LL^{2}(
\Omega\times\mathbb{R}^{d})  }ds\\
&    \leq  C \ \int_{0}^{t}\big\Vert (  I-A  )^{1/2}
e^{\left(  t-s\right) \mathcal{L}}\big\Vert _{\LL^{2}
\rightarrow \LL^{2}  }\ \big\Vert \left(  I-A\right)^{\epsilon/2}   e^{h\mathcal{L}}\left(
 V^{N} \ast\left(  F(\cdot\; ,\;g_{s}^{N})S_{s}^{N}\right)
\right)  \big\Vert _{\LL^{2}(  \Omega\times\mathbb{R}^{d})  }ds.
\end{align*}
 We have 
\begin{equation*}
 \big\Vert \left(  I-A\right)^{1/2}e^{\left(  t-s\right)\mathcal{L}} \big\Vert _{\LL^{2}  \rightarrow \LL^{2} } \leq \frac{C}{(t-s)^{ 1/\alpha}},
\end{equation*}
from Proposition \ref{prop}. 
On the other hand, for any $x \in \RR^d$, 
\[
\big\vert \left(  V^{N}\ast\left(  F(\cdot\; , \;g_{s}^{N})S_{s}^{N}\right)  \right)
\left(  x\right)  \big\vert \leq\left\Vert F\right\Vert _{\infty}\big|V^{N} \ast S_{s}^{N}  \left(  x\right)\big|
=\left\Vert F\right\Vert _{\infty} \; |g_{s}^{N}\left(  x\right)|.
\]
Then by Lemma \ref{positi} we have 
\[
\left\Vert   \left(  I-A\right)^{\epsilon/2} e^{h\mathcal{L}}\Big[V^{N}\ast\left(  F(\cdot\; , \;g_{s}^{N})S_{s}^{N}\right)
\Big]\right\Vert _{\LL^{2}(  \Omega\times\mathbb{R}^{d})  }\leq
C(F)\; \big\Vert  \left(  I-A\right) ^{\epsilon/2}  e^{h\mathcal{L}}g_{s}^{N}\big\Vert _{\LL^{2}(
\Omega\times\mathbb{R}^{d})  }.
\]
To summarize, we have%
\begin{multline*}
  \int_{0}^{t}  \big\Vert \left(  I-A\right)  ^{\epsilon/2}\nabla
e^{\left(  t+h-s\right) \mathcal{L}}\left(   V^{N} \ast\left(
F(\cdot\; , \; g_{s}^{N})S_{s}^{N}\right)  \right)  \big\Vert _{\LL^{2}(
\Omega\times\mathbb{R}^{d})  }ds\\
 \leq\int_{0}^{t}\frac{C}{\left(  t-s\right) ^{1/\alpha}}\big\Vert  \left(  I-A\right)  ^{\epsilon/2}  e^{h\mathcal{L}}g_{s}^{N}\big\Vert _{\LL^{2}(
\Omega\times\mathbb{R}^{d})  }ds.
\end{multline*}
\textsc{Step 4}. The estimate of the third term \eqref{eq:3} is quite tricky and we
postpone it to Lemma   \ref{lemma martingale 1} below, where we prove that $(I-A)^{\varepsilon /2}\; M_t^N$ is uniformly bounded in $\LL^2(\Omega\times \RR^d)$. Collecting the three
bounds together, we have%
\[
\left\Vert \left(  I-A\right)  ^{\epsilon/2}e^{h\mathcal{L}}g%
_{t}^{N}\right\Vert _{\LL^{2}(  \Omega\times\mathbb{R}^{d})  }%
\leq C  +\int_{0}^{t}\frac{C}{\left(  t-s\right)
^{1 /\alpha}}\big\Vert \left(  I-A\right)
^{\epsilon/2}e^{h\mathcal{L}}g_{s}^{N}\big\Vert _{\LL^{2}(
\Omega\times\mathbb{R}^{d})  }ds.
\]
We may apply a generalized version of Gronwall lemma 
and conclude 
\[
\big\Vert \left(  I-A\right)  ^{\epsilon/2}e^{h\mathcal{L}}g%
_{t}^{N}\big\Vert _{\LL^{2}(  \Omega\times\mathbb{R}^{d})  }%
\leq C.
\]
We may now take the limit as $h\rightarrow0$. The proof is complete.
\end{proof}

\begin{remark}
Taking $\varepsilon=0$ we get that $g_t^N$ is also uniformly bounded: for any $t \in [0,T]$ and $N \in \mathbb{N}$, 
\begin{equation}
\label{eq:L2g} \mathbb{E}\big[ \Vert g_t^N \Vert_{\LL^2}\big] \leq C.
\end{equation}
\end{remark}

\subsubsection{Martingale estimate}

\begin{lemma} \label{lemma martingale 1} For any $N\in\mathbb N$ it holds
\[
\big\Vert  \left(  I-A\right)^{\epsilon/2} M_t^{N}    \big\Vert _{\LL^{2}(
\Omega\times\mathbb{R}^{d})}  \leq C. 
\]
\end{lemma}

\begin{proof} \textsc{Step 1}. We decompose $M_t^N$ according to the sum of two integrals, over $|z|\leq 1$ and $|z| \geq 1$, as follows: 
\begin{align*}
& M_t^{N}
 = I_1 + I_2 \\
%&  \frac{1}{N}\sum_{i=1}^{N} \int_{0}^{t} \int_{\RR^{d}-\{0\}}   \big(
%e^{(t-s+ h)L} V^{N} \big)  \big(  x-X_{s_{-}}^{i,N} + z  \big)  -   
%   \big(e^{(t-s+ h)\mathcal{L}} V^{N} \big)\big( x-X_{s_{-}}^{i,N}\big)  d\tilde{\mathcal N}^{i}(dsdz)\\
 & = \frac{1}{N}\sum_{i=1}^{N} \int_{0}^{t} \int_{|z|\geq 1}  \Big\{ \big(
e^{(t-s+ h)\mathcal{L}} V^{N} \big)  \big(  x-X_{s_{-}}^{i,N} + z  \big)  -   
   \big(e^{(t-s+ h)\mathcal{L}} V^{N}\big)\big( x-X_{s_{-}}^{i,N}\big) \Big\} d\tilde{\mathcal N}^{i}(dsdz)\\
& +   \frac{1}{N}\sum_{i=1}^{N} \int_{0}^{t} \int_{|z|\leq1}  \Big\{ \big(
e^{(t-s+ h)\mathcal{L}} V^{N} \big)  \big(  x-X_{s_{-}}^{i,N} + z  \big)  -   
   \big(e^{(t-s+ h)\mathcal{L}} V^{N} \big)\big( x-X_{s_{-}}^{i,N}\big)\Big\}  d\tilde{\mathcal N}^{i}(dsdz). 
\end{align*}
When $|z|\geq 1$ we can use the fact that $\int_{|z|\geq 1} d\nu(z) < \infty$. When $|z|\leq 1$ this is not true, but we know that $\int_{|z|\leq 1} |z|^2 d\nu(z) < \infty$. Therefore, for this second case, we use the tool of the maximal function introduced in Section \ref{sec:max}. For the sake of clarity, along this proof we denote
\begin{equation}\label{eq:fn}
 F^N(s,x,z) := V^{N}  \big(  x-X_{s_{-}}^{i,N} + z  \big)  -   
   V^{N}\big( x-X_{s_{-}}^{i,N}\big).
\end{equation}

\medskip

\noindent 
\textsc{Step 2}. Let us start with $I_1$. We have 
\begin{align*}
  \big\Vert (  I&-A)^{\epsilon/2}  I_1\big\Vert _{\LL^{2}(  \Omega\times
\mathbb{R}^{d})  }^{2}\\
&  =
\frac{2}{N^{2}} \int_{\mathbb{R}^{d}}\mathbb{E}\Bigg[  \Bigg\vert \sum_{i=1}^{N} \int_{0}^{t}
\int_{|z|\geq 1}  \left(  I-A\right)^{\epsilon/2}  e^{(t-s+ h)\mathcal{L}} F^N(s,x,z) 
 d\tilde{\mathcal N}^{i}(dsdz)   \Bigg\vert ^{2}\Bigg]  dx\\
&  =\frac{2}{N^{2}}     \sum_{i=1}^{N} \int_{\mathbb{R}^{d}}\mathbb{E}\left[  \int_{0}^{t}
\int_{|z|\geq 1} \Big| \left(  I-A\right)^{\epsilon/2}  e^{(t-s+ h)\mathcal{L}} F^N(s,x,z) \Big|^{2}
  d\nu(z) ds    \right]  dx \\
&  = \frac{2}{N}  \;   \mathbb{E} \left[  \int_{0}^{t} \int_{|z|\geq 1}  \int_{\mathbb{R}^{d}}
 \Big| \left(  I-A\right)^{\epsilon/2}  e^{(t-s+ h)\mathcal{L}} F^N(s,x,z) \Big|^{2} \ dx \      d\nu(z) ds    \right].  
\end{align*}
 We observe  that 
\begin{align*}
 \int_{\mathbb{R}^{d}}
 \Big| \left(  I-A\right)^{\epsilon/2}  e^{(t-s+ h)\mathcal{L}} F^N(s,x,z) \Big|^{2} \ dx &\leq 2 \int_{\mathbb{R}^{d}}
 \Big| \left(  I-A\right)^{\epsilon/2}  e^{(t-s+ h)\mathcal{L}}  V^{N}(x) \Big|^{2} \ dx  \\
 &\leq  C \big\|  V^{N} \big\|_{\HH^{\epsilon}}^{2}.
\end{align*}
This implies that
\[
 \big\Vert \left(  I-A\right)^{\epsilon/2}  I_1\big\Vert _{\LL^{2}(  \Omega\times
\mathbb{R}^{d})  }^{2}\leq 
\frac{C}{N}  \big\| V^{N} \big\|_{\HH^{\epsilon}}^{2} \leq
\frac{C N^{ \beta + \frac{2 \epsilon\beta}{d}} \|V\|^2_{\HH^\epsilon}}{N} \leq C,
\]
where in the last inequality we use Assumption \ref{assump}. 

\medskip 

\noindent \textsc{Step 3}. In the same way we obtain 
\begin{align*}
 \big\Vert (  I&-A)^{\epsilon/2}  I_2\big\Vert _{\LL^{2}(  \Omega\times
\mathbb{R}^{d})  }^{2}\\
&  =
\frac{2}{N^{2}} \int_{\mathbb{R}^{d}}\mathbb{E}\Bigg[  \Bigg\vert \sum_{i=1}^{N} \int_{0}^{t}
\int_{|z|\leq 1 }  \left(  I-A\right)^{\epsilon/2}  e^{(t-s+ h)\mathcal{L}} F^N(s,x,z) 
 d\tilde{\mathcal N}^{i}(dsdz)   \Bigg\vert ^{2}\Bigg]  dx\\
&  =\frac{2}{N^{2}}     \sum_{i=1}^{N} \int_{\mathbb{R}^{d}}\mathbb{E}\left[  \int_{0}^{t}
\int_{|z|  \leq 1} | \left(  I-A\right)^{\epsilon/2}  e^{(t-s+ h)\mathcal{L}} F^N(s,x,z) |^{2}
 d\nu(z) ds     \right]  dx \\
&  = \frac{2}{N^{2}}     \sum_{i=1}^{N}\mathbb{E} \left[  \int_{0}^{t} \int_{|z|\leq  1}  \int_{\mathbb{R}^{d}}
 | \left(  I-A\right)^{\epsilon/2}  e^{(t-s+ h)\mathcal{L}} F^N(s,x,z) |^{2} \ dx    d\nu(z) ds    \right]  
\end{align*}
where $F^N$ has been defined in \eqref{eq:fn}.
We observe that 
\begin{align*}
(  I-A)^{\epsilon/2}  e^{(t-s+ h)\mathcal{L}} F^N(s,x,z) 
\leq &\;  C |z|\; \mathbb{M} \Big[ \nabla(I-A)^{\epsilon/2} \; e^{(t-s+ h)\mathcal{L}} V^N\big(x-X_{s_{-}}^{i,N} + z \big)\Big]\\
 & + C |z| \; \mathbb{M}\Big[ \nabla(I-A)^{\epsilon/2} \; e^{(t-s+ h)\mathcal{L}} V^N\big(x-X_{s_{-}}^{i,N}\big) \Big].
\end{align*}
From  Lemma \ref{maxi} we have  
\[
\Big\|  \mathbb{M} \Big[\nabla(I-A)^{\epsilon/2}  e^{(t-s+ h)\mathcal{L}} V^{N} \Big] \Big\|_{\LL^{2}}^{2}\leq C
\Big\|\nabla  (I-A)^{\epsilon/2}  e^{(t-s+ h)\mathcal{L}} V^{N} \Big\|_{\LL^{2}}^{2}.
\]
This implies that 
\[
\Big\| (  I-A)^{\epsilon/2}  e^{(t-s+ h)\mathcal{L}} F^N(s,x,z)  \Big\|_{\LL^2}^{2} \leq C \; |z|^2
\; \Big\|\nabla  (I-A)^{\epsilon/2}  e^{(t-s+ h)\mathcal{L}} V^{N} \Big\|_{\LL^{2}}^{2}.
\]
Taking $\delta\in \big(1-\frac{\alpha}{2}, \frac{(1-\beta)d}{2\beta}-\epsilon\big]$ (which exists by Assumption \ref{assump})  we have
\begin{align*}
 \big\Vert \left(  I-A\right)^{\epsilon/2}  I_2\big\Vert _{\LL^{2}(  \Omega\times
\mathbb{R}^{d})  }^{2}&  \leq    \frac{C}{N}  \int_{0}^{t} \big\|\nabla  (I-A)^{\epsilon/2}  e^{(t-s+ h)\mathcal{L}} V^{N} \big\|_{\LL^{2}}^{2}
 ds  \\
&  \leq    \frac{C}{N}   \big\| (I-A)^{(\epsilon +\delta)/2} V^{N}\big\|_{\LL^{2}}^{2} 
 \int_{0}^{t} \frac{1}{(t+h-s)^{2(1-\delta)/\alpha}} ds \leq C
\end{align*}
where we used that        
\[
\big\Vert V^{N}\big\Vert _{\HH^{\epsilon + \delta}}^{2}\leq C  N^{\beta+ \frac{2 (\epsilon+\delta)\beta}{d}  } \big\Vert V \big\Vert
_{\HH^{\epsilon + \delta}}^{2} \leq C
\]
and $ \|V\|_{\HH^{\epsilon + \delta}} < \infty$ from Assumption \ref{assump}. 
\end{proof}

\subsubsection{Second  estimate on $g^{N}$}

\begin{lemma}\label{estimation2} For any $\gamma  \in(0,\frac{1}{2})$ and $N\in\mathbb{N}$, it holds
\[
\EE\left[ \int_{0}^{T}\int_{0}^{T}  \frac{ \big\|g_t^{N}-g_s^{N}\big\|_{\HH^{-2}}^{2}}{|t-s|^{1+ 2\gamma}}      \ ds \ dt \right]  \leq C.
\]

\end{lemma}

\begin{proof}  In this proof we use the fact that $\LL^2(\RR^d) \subset \mathbb{W}^{-2,2}(\RR^d)$ with continuous embedding, and that the linear operator $\Delta$ is bounded from $\LL^2(\RR^d)$ to $\mathbb{W}^{-2,2}(\RR^d)$.

\medskip

\noindent \textsc{Step 1}.  We recall the formula
\begin{align*}
\big(  e^{h\mathcal{L}}&g_{t}^{N}\big)  (x)    =\big(  e^{h\mathcal{L}}g_{0}^{N}\big)  (x)+\int_{0}^{t}\left\langle S_{s}%
^{N},F(\cdot \; ,\; g_{s}^{N})\nabla\big(  e^{h\mathcal{L}} V^{N}
\big)  \left(  x-\cdot\right)  \right\rangle ds\\
&  +\frac{1}{2}\int_{0}^{t}\mathcal{L} \big(  e^{h\mathcal{L}}g_{s}%
^{N}\big)  (x)ds\\
& + \frac{1}{N}\sum_{i=1}^{N} \int_{0}^{t} \int_{\RR^{d}-\{0\}}  \Big\{ \big(
e^{h\mathcal{L}} V^{N} \big)  \big(  x-X_{s_{-}}^{i,N}+ z  \big)  -   
   \big(e^{h\mathcal{L}} V^{N}\big)\big( x-X_{s_{-}}^{i,N}\big)\Big\}  d\tilde{\mathcal N}^{i}(dsdz).
\end{align*}
Then we have 
\begin{align*}
\big( e^{h\mathcal{L}} &g_{t}^{N}\big)  (x) - \big( e^{h\mathcal{L}}g_{s}^{N}\big)  (x) = 
\int_{s}^{t}\left\langle S_{r}%
^{N},F(\cdot \; ,\; g_{r}^{N})\nabla\big(  e^{h\mathcal{L}} V^{N}
\big)  \left(  x-\cdot\right)  \right\rangle dr\\
&  +\frac{1}{2}\int_{s}^{t}\mathcal{L} \big(  e^{h\mathcal{L}}g_{r}%
^{N}\big)  (x)dr\\
& + \frac{1}{N}\sum_{i=1}^{N} \int_{s}^{t} \int_{\RR^{d}-\{0\}}\Big\{   \big(
e^{h\mathcal{L}} V^{N} \big)  \big(  x-X_{r_{-}}^{i,N}+ z  \big)  -   
   \big(e^{h\mathcal{L}} V^{N}\big)\big( x-X_{r_{-}}^{i,N}\big) \Big\} d\tilde{\mathcal N}^{i}(drdz).
\end{align*}
Thus by letting $h\to 0$ we obtain, for $s <t$,  
\begin{align*}
 \EE \Big[ & \big\| g_{t}^{N}-g_{s}^{N}  \big\|_{\HH^{-2}}^{2}\Big] \\ 
 & \leq (t-s)  \int_{s}^{t} \EE \Big[ \big\|\nabla
\left(  V^{N} \ast\left(
F(\cdot \; ,\; g_{r}^{N})S_{r}^{N}\right)  \right)\big\|_{\HH^{-2}}^{2}\Big]  dr     +   \frac{t-s}{2}\int_{s}^{t}  \EE \Big[ \big\|\mathcal{L} g_{r}^{N} \big\|_{\HH^{-2}}^{2}\Big] dr
\\ & 
+   \EE\bigg[\bigg\| \frac{1}{N}\sum_{i=1}^{N} \int_{s}^{t} \int_{\RR^{d}-\{0\}}   \Big\{( V^{N})  \big(  x-X_{r_{-}}^{i,N} + z \big) -   
   ( V^{N})\big( x-X_{r_{-}}^{i,N}\big) \Big\} d\tilde{\mathcal N}^{i}(drdz)\bigg\|_{\HH^{-2}}^{2}\bigg].
\end{align*}

\noindent \textsc{Step 2}. We observe that 
\begin{align*}
 \int_{s}^{t}  \EE \Big[\Big\|\nabla \left(  V^{N} \ast\left(
F(\cdot\; ,\; g_{r}^{N})S_{r}^{N}\right)  \right)\Big\|_{\HH^{-2}}^{2}\Big]  dr   
&\leq  \int_{s}^{t} \EE \Big[\Big\|
\left(  V^{N} \ast\left(
F(\cdot \; ,\; g_{r}^{N})S_{r}^{N}\right)  \right)\Big\|_{\HH^{-1}}^{2}\Big]  dr
\\ & 
\leq  \int_{s}^{t}  \Big\|
\left(  V^{N} \ast\left(
F(\cdot\; ,\; g_{r}^{N})S_{r}^{N}\right)  \right)\Big\|_{\LL^{2}(\Omega\times \RR)}^{2}  dr
\\ & 
\leq  \int_{s}^{t}  \big\| g_r^{N}\big\|_{\LL^{2}(\Omega\times \RR)}^{2}  dr \leq C (t-s),
\end{align*}
where in last step we used Lemma  \ref{estimation} (see \eqref{eq:L2g}).

\medskip
\noindent \textsc{Step 3}. Moreover
\begin{align*}
  \EE\bigg[\bigg\| \frac{1}{N}\sum_{i=1}^{N} & \int_{s}^{t} \int_{\RR^{d}-\{0\}}   \big(
 V^{N} \big)  \big(  \cdot -X_{r_{-}}^{i,N} + z \big) -   
   \big( V^{N} \big)\big(  \cdot -X_{r_{-}}^{i,N}\big)  d\tilde{\mathcal N}^{i}(drdz)\bigg\|_{\HH^{-2}}^{2} \bigg]\\
& =   \frac{1}{N^{2}}\sum_{i=1}^{N} \int_{s}^{t} \int_{\RR^{d}-\{0\}}  \big\| \big(
 V^{N} \big)  \big(   \cdot -X_{r_{-}}^{i,N} + z \big) -   
   \big( V^{N} \big)\big(  \cdot -X_{r_{-}}^{i,N}\big) \big\|_{\HH^{-2}}^{2}\; dr d\nu(z) \\ 
   & =   \frac{1}{N^{2}}\sum_{i=1}^{N} \int_{s}^{t} \int_{|z|\geq 1}  \big\| \big(
 V^{N} \big)  \big(   \cdot-X_{r_{-}}^{i,N} + z \big) -   
   \big( V^{N} \big)\big(  \cdot -X_{r_{-}}^{i,N}\big) \big\|_{\HH^{-2}}^{2}\; dr d\nu(z) \\ 
& +   \frac{1}{N^{2}}\sum_{i=1}^{N} \int_{s}^{t} \int_{|z|\leq 1}  \big\| \big( V^{N} \big)  \big(   \cdot -X_{r_{-}}^{i,N} + z \big) -   
   \big( V^{N} \big)\big(  \cdot -X_{r_{-}}^{i,N}\big) \big\|_{\HH^{-2}}^{2}\; dr d\nu(z).
\end{align*}

\noindent \textsc{Step 4}. We observe 
\begin{multline*} 
   \frac{1}{N^{2}}\sum_{i=1}^{N}  \int_{s}^{t} \int_{|z|\geq 1}  \big\| \big(
 V^{N} \big)  \big(  \cdot -X_{r_{-}}^{i,N} + z \big) -   
  \big( V^{N} \big)\big(  \cdot -X_{r_{-}}^{i,N}\big)\big \|_{\HH^{-2}}^{2}\;  dr d\nu(z) \\
  \leq \frac{2}{N^{2}}\sum_{i=1}^{N} \int_{s}^{t} \int_{|z|\geq 1}  \big\| 
 V^{N}   \big\|_{\HH^{-2}}^{2}\;  dr d\nu(z)   \leq \frac{C}{N} \int_{s}^{t}   \big\|  V^{N}    \big\|_{\LL^2}^{2}\;  dr \leq C (t-s),
\end{multline*}
where in the last step we use the fact that $\|V^N\|_{\LL^2}^2 \leq N^\beta \|V\|_{\LL^2}^2$. 

\medskip 

\noindent \textsc{Step 5}. Let us write
\begin{align*}
   \frac{1}{N^{2}}\sum_{i=1}^{N}&  \int_{s}^{t} \int_{|z|\leq 1}   \big\| \big(
 V^{N} \big)  \big(   \cdot -X_{r_{-}}^{i,N} + z \big) -   
   \big(V^{N} \big)\big(  \cdot -X_{r_{-}}^{i,N}\big) \big\|_{\HH^{-2}}^{2} \; dr d\nu(z) \\ 
& \leq   \frac{C}{N^{2}}\sum_{i=1}^{N} \int_{s}^{t} \int_{|z|\leq 1}  \big\| \big(
 V^{N} \big)  \big(   \cdot -X_{r_{-}}^{i,N} + z \big) -   
  \big( V^{N}\big)\big(  \cdot -X_{r_{-}}^{i,N}\big) \big\|_{\HH^{-2}}^{2}\; dr d\nu(z) \\
   & \leq   \frac{C}{N^{2}}\sum_{i=1}^{N} \int_{s}^{t} \int_{|z|\leq 1} |z|^{2}\; \big\| \mathbb{M}\big(
 V^{N}\big) \big\|_{\HH^{-2}}^{2} \;dr d\nu(z)  \\
& \leq   \frac{C}{N} \int_{s}^{t}\big\| \nabla \big(
 V^{N} \big) \big\|_{\HH^{-2}}^{2}\; dr   \leq   \frac{C}{N} \int_{s}^{t}\big\|
 V^{N}  \big\|_{\LL^{2}}^{2} \; dr\leq  C(t-s).
\end{align*}
\textsc{Step 6}.  Finally we  claim that 
\[
\frac{1}{2}\int_{s}^{t}  \EE  \Big[\big\|\mathcal{L} g_{r}^{N}\big\|_{\HH^{-2}}^{2}\Big] dr
\leq \frac{C}{2} \int_{s}^{t}  \EE \Big[\big\| g_{r}^{N}\big\|_{\LL^{2}}^{2} \Big] dr\leq  C (t-s). 
\]
The first inequality above can be proved as follows:  
\begin{align}
\big \|\mathcal{L} g_{r}^{N}\big\|_{\HH^{-2}}&=\sup_{\| f\|_{\HH^2}\leq 1}|<\mathcal{L} g_{r}^{N},f>|
\notag \\ & 
=\sup_{\| f\|_{\HH^2}\leq 1}|< g_{r}^{N},\mathcal{L} f>|\leq \sup_{\| f\|_{\HH^2}\leq 1} \Big\{  \| g_{r}^{N}\big\|_{\LL^2} \  \|\mathcal{L} f \big\|_{\LL^2}\Big\}
 \notag \\ &
\leq  \  C \sup_{\| f\|_{\HH^2}\leq 1}\Big\{ \| g_{r}^{N}\big\|_{2} \big( \|\Delta f \big\|_{\LL^2} + \|\nabla f \big\|_{\LL^2} \big)\Big\} \label{ineq} \\ & \leq  C \ \| g_{r}^{N}\big\|_{\LL^2}, \notag
\end{align}
 where  we used \cite[Lemma 2.4]{Zhang2} to obtain  \eqref{ineq}. 
\end{proof}

\subsection{Criterion of compactness} 
\label{subsect compactness}

In this  subsection we follow  the arguments of  \cite[Section 6.1]{flan}.  We start by constructing one space on which the sequence of the laws of $g_\cdot^N$ will be tight.

 A version of the Aubin-Lions Lemma, see 
\cite{FlaGat,Lions}, states that when $E_{0}\subset E\subset E_{1}$ are three
Banach spaces with continuous dense embeddings, with $E_{0},E_{1}$ reflexive, and 
$E_{0}$ compactly embedded into $E$, given $p,q\in(  1,\infty)  $
and $\gamma\in(  0,1)  $, the space $\LL^{q}( [ 0,T]\; ;\; E_{0})
\cap \mathbb{W}^{\gamma,p}( [ 0,T] \; ;\; E_{1})  $ is compactly embedded into
$\LL^{q}(  [0,T] \; ;\; E)  $.

We use this lemma with $E=\mathbb{W}^{\eta,2}(  D)$,
$E_{0}=\mathbb{W}^{\epsilon,2}(  D)$ with
$
\frac{d}{2}<\eta<\epsilon,
$ and $E_{1}=\mathbb{W}^{-2,2}(  \mathbb{R}^{d})$ 
where $D$ is a regular bounded domain. We also choose  
 $0<\gamma<\frac{1}{2}$ in order to apply Lemma \ref{estimation2} (see below). The Aubin-Lions Lemma states that \[\LL^{2}\big([0,T]\; ;\; \mathbb{W}^{\epsilon,2}(  D) \big)  \cap \mathbb{W}^{\gamma,2}\big(
[0,T]\; ;\; \mathbb{W}^{-2,2}(  \mathbb{R}^{d}) \big)  \] is compactly embedded
into $\LL^{2}\big(  [0,T]\; ;\; \mathbb{W}^{\eta,2}(  D) \big)  $. Now, consider the space
\[
Y_{0}:=\LL^{\infty}\big(  [0,T]\; ;\; \LL^{2}(  \mathbb{R}^{d})  \big)  \cap
\LL^{2}\big(  [0,T]\; ;\; \mathbb{W}^{\epsilon,2}(  \mathbb{R}^{d}) \big)  \cap
\mathbb{W}^{\gamma,2}\big(  [0,T]\; ; \; \mathbb{W}^{-2,2}(  \mathbb{R}^{d}) \big)  .
\]
Using the Fr\'{e}chet topology on $\LL^{2}\big( [ 0,T]\; ;\; \mathbb{W}_{\rm loc}^{\eta,2}(  \mathbb{R}^{d})  \big)  $ defined as
\[
d\left(  f,g\right)  =\sum_{n=1}^{\infty}2^{-n}\left(  1\wedge\int_{0}%
^{T}\big\Vert f (  t,\cdot )  \big\Vert _{\mathbb{W}^{\eta
,2} (  \mathcal{ B}(  0,n )  )  }^{2}dt\right)
\]
one has that $\LL^{2}\big(  [0,T]\; ;\; \mathbb{W}^{\epsilon,2}(  \mathbb{R}^{d})
\big)  \cap \mathbb{W}^{\gamma,2}\big(  [0,T]\; ;\; \mathbb{W}^{-2,2}(  \mathbb{R}^{d})
\big)  $ is compactly embedded\footnote{The proof is
elementary, using the fact that if a set is compact in $\LL^{2}\big(
[0,T]\; ;\; \mathbb{W}_{\rm loc}^{\eta,2} (  \mathcal{B} (  0,n )   ) \big)
$ for every $n$ then it is compact in $\LL^{2}\big(  [0,T]\; ;\; \mathbb{W}_{\rm loc}^{\varepsilon,2}(  \mathbb{R}^{d})  \big)  $ with this topology.} into $\LL^{2}\big( [ 0,T]\; ;\; \mathbb{W}_{\rm loc}^{\eta,2}(  \mathbb{R}^{d})  \big)$. 

Let us denote respectively by $\LL_{w\ast}^{\infty} $ and $\LL_{w}^{2}$ the  spaces
$\LL^{\infty} $ and
$\LL^{2}  $
endowed respectively with the weak star and weak topology. We have that
$Y_{0}$ is compactly embedded into%
\begin{equation}
Y:=\LL_{w\ast}^{\infty}\big([  0,T]\; ;\; \LL^{2}(  \mathbb{R}^{d})  \big)
\cap \LL_{w}^{2}\big( [ 0,T]\; ;\; \mathbb{W}^{\epsilon,2}(  \mathbb{R}^{d})  \big)
\cap \LL^{2}\big(  [0,T]\; ;\; \mathbb{W}_{\rm loc}^{\eta,2}(  \mathbb{R}^{d})  \big). \label{topology of convergence}%
\end{equation}
Note that%
\[
\LL^{2}\big( [ 0,T]\; ;\; \mathbb{W}_{\rm loc}^{\eta,2}(  \mathbb{R}^{d})
\big)  \subset \LL^{2}\big(  [0,T]\; ;\; C(  D)  \big)
\]
for every regular bounded domain $D\subset\mathbb{R}^{d}$.

\medskip

 Let us now go back to the sequence of processes $\{g_\cdot^N\}_N$, for which we have proved several estimates. The  Chebyshev inequality ensures that 
\[
\PP\big( \|g_{\cdot}^{N} \|_{Y_{0}}^{2} > R\big)\leq \frac{\EE \big[\| g_{s}^{N}\|_{Y_{0}}^{2}\big]}{R}, \qquad \text{for any } R>0.
\]
Thus by Lemma \ref{estimation} and Lemma \ref{estimation2} we obtain 
\[
\PP\big( \|g_{\cdot}^{N} \|_{Y_{0}}^{2} > R\big)\leq \frac{C}{R}, \qquad \text{for any } R>0, N \in \NN.
\]
The process $(g_{t}^{N})_{t\in[0,T]}$ defines a probability $\P_{N}$ on $Y$. Last inequality implies that there exists a bounded set $B_{\epsilon}\in Y_{0}$ 
such that $\P_{N}(B_{\epsilon})< 1-\epsilon$ for all $N$, and therefore there exists a compact set $
K_{\epsilon}\in Y$ such  that $\P_{N}(K_{\epsilon})< 1-\epsilon$.

Denote by $\{  \texttt{L}^{N}\}_{N\in\mathbb{N}}$ the laws of the processes $\{
g^{N}\}  _{N\in\mathbb{N}}$ on $Y_{0}$,  we have proved that $\{
\texttt{L}^{N}\}  _{N\in\mathbb{N}}$ is tight in $Y$, hence relatively compact,
by Prohorov's Theorem. From every subsequence of $\{  \texttt{L}^{N}\}_{N\in\mathbb{N}}$ it is possible to extract a further subsequence which
converges to a probability measure $\texttt{L}$ on $Y$. Moreover  by a Theorem of
Skorokhod (see \cite[Theorem 2.7]{Ikeda}), we are allowed, eventually after choosing a suitable probability space  where all our random variables can be defined, to assume
\begin{equation*}
g^{N} \rightarrow \ u\quad \text{in }   Y, \qquad \text{a.s.}
\end{equation*}
where the law of $u$ is $\texttt{L}$.

\subsection{Passing to the limit} \label{sec:limit}

In this paragraph we show that the limit $u$ of $g^N$ satisfies the weak formulation \eqref{limiteq} of the non-local conservation equation \eqref{PDEintr}. 

\textsc{Step 1}. By It\^o's formula we have, for any test function $\phi$, that 
%\begin{align*}
%g_{t}^{N}(x)  &  =\ g_{0}^{N} (x)+\int_{0}^{t}\left\langle S_{s}%
%^{N},F(\cdot\; ,\; g_{s}^{N})\cdot\nabla V^{N} \left(  x-\cdot\right)  \right\rangle ds\\
%&  +\frac{1}{2}\int_{0}^{t}L g_{s}
%^{N}(x)ds\\
%& + \frac{1}{N}\sum_{i=1}^{N} \int_{0}^{t} \int_{\RR^{d}-\{0\}}    V^{N}   \big(  x-X_{s_{-}}^{i,N}+ z  \big)  -   
%    V^{N}\big( x-X_{s_{-}}^{i,N}\big)  d\tilde{N}^{i}(dsdz)
%\end{align*}
%Applying  test function we  have 
\begin{align*}
\int g_{t}^{N}(x) & \phi(x) \;  dx     =\ \int  g_{0}^{N} (x)  \phi(x) \;  dx +
\int_{0}^{t}\Big\langle S_{s}%
^{N},F(\cdot\; ,\; g_{s}^{N})\nabla (V^{N}\ast\phi) \left(\cdot\right)  \Big\rangle \; ds\\
&  +\frac{1}{2}\int_{0}^{t}  \int  g_{s}
^{N}(x)    \mathcal{L}\phi(x) \;  dx  ds\\
& + \frac{1}{N}\sum_{i=1}^{N} \int_{0}^{t} \int_{\RR^{d}-\{0\}}   \Big\{   (V^{N}\ast\phi) \big(  -X_{s_{-}}^{i,N}+ z  \big)  -   
    (V^{N}\ast\phi) \big( -X_{s_{-}}^{i,N}\big) \Big\} \; d\tilde{\mathcal N}^{i}(dsdz).
\end{align*}
Passing to the limit we obtain 
\begin{align*}
\int &u(t,x)  \phi(x) \;  dx     = \int  u_{0}(x)  \phi(x) \;  dx +
\lim_{N\rightarrow \infty}\int_{0}^{t}\Big\langle S_{s}%
^{N},F(\cdot \; ,\; g_{s}^{N})\nabla  (V^{N}\ast\phi) \left(\cdot\right)  \Big\rangle ds\\
&  +\frac{1}{2}\int_{0}^{t}  \int  u(s,x)    \mathcal{L}\phi(x) \;  dx  ds\\
& + \lim_{N\rightarrow \infty} \frac{1}{N}\sum_{i=1}^{N} \int_{0}^{t} \int_{\RR^{d}-\{0\}}     \Big\{  (V^{N}\ast\phi) \big(-X_{s_{-}}^{i,N}+ z  \big)  -   (V^{N}\ast\phi) \big(-X_{s_{-}}^{i,N}\big)  \Big\}\;  d\tilde{\mathcal N}^{i}(dsdz).
\end{align*}
\textsc{Step 2}. We claim that 
\begin{equation}\label{eq:claim}
\lim_{N\rightarrow \infty}\int_{0}^{t}\Big\langle S_{s}%
^{N},F(\cdot \; ,\; g_{s}^{N})\nabla  (V^{N}\ast\phi) \left(\cdot\right)  \Big\rangle ds= \int_{0}^{t} \int u(s,x) F(x,u) \nabla \phi(x) \; dx  ds.
\end{equation} 
In order to prove the claim, we first observe, using the symmetry of $V$, that 
\begin{multline*}
\Big| \Big\langle S_{s}^{N},F(\cdot \; ,\; g_{s}^{N})\nabla  (V^{N}\ast\phi) \left(\cdot\right)  \Big\rangle
-\Big\langle g_{s}^{N},F(\cdot \; ,\; g_{s}^{N})\nabla  (V^{N}\ast\phi) \left(\cdot\right)  \Big\rangle \Big|
\\
\leq \sup_{x\in \RR^{d} }\Big| F(x ,\; g_{s}^{N}(x))\nabla  (V^{N}\ast\phi)(x)  
- \big( (F(\cdot \; ,\; g_{s}^{N})\nabla  (V^{N}\ast\phi) \left(\cdot\right)) \ast  V^{N} \big)(x) \Big|,
\end{multline*}
 We can control the last term, using the fact that \begin{itemize} \item $V$ is a density (denoted below by $(\int V =1)$), 
 \item  $F$ is Lipschitz and bounded (denoted below by ($F\in$ Lip$\cap L^\infty$)), \item $V$ is compactly supported (denoted below by ($V$ is c.s.)),
 \item and $\phi$ is compactly supported and smooth, 
  \end{itemize}  as follows:
\begin{align*}
  \Big|F(&x ,\;g_{s}^{N}(x))\nabla  (V^{N}\ast\phi)(x)  
- \big( (F(\cdot \; ,\; g_{s}^{N})\nabla (V^{N}\ast\phi) \left(\cdot\right)) \ast V^{N} \big)(x) \Big| \vphantom{\Bigg\{}\\
&
 \overset{\substack{\hphantom{(F\in\mathrm{Lip}\cap L^\infty)}\\(\int V=1)}}{\leq}   \int V(y) \ \big|\nabla (V^{N}\ast\phi)(x)\big| \
 \Big| F\big( x,  g_{s}^{N}(x)\big)-F\Big( x -\frac{y}{N^{\frac{\beta}{d}}} , 
g_{s}^{N}\Big(x-\frac{y}{N^{\frac{\beta}{d}}} \Big)\Big)\Big| dy \notag \\
 & \qquad \qquad
 +    \int V(y)  \ \Big|\nabla (V^{N}\ast\phi)(x)-\nabla(V^{N}\ast\phi)\Big(x- \frac{y}{N^{\frac{\beta}{d}}} \Big)  \Big| \ 
\big|F( x,  g_{s}^{N}(x))\big|dy 
 \\ &
\overset{(F\in\mathrm{Lip}\cap L^\infty)}{\leq}  C \int V(y) \ \big|\nabla (V^{N}\ast\phi)(x)\big| \
 \Big| g_{s}^{N}(x)-g_{s}^{N}\Big(x-\frac{y}{N^{\frac{\beta}{d}}} \Big)\Big| dy \\ 
 & \qquad \qquad 
 +  \frac{C}{N^{\frac{\beta}{d}}}  \int V(y) |y|     dy \\
& 
 \overset{\substack{\hphantom{(F\in\mathrm{Lip}\cap L^\infty)}\\(V \mathrm{\; is  \; c.s.})}}{\leq}   \frac{C}{N^{\frac{\tilde\eta \beta}{d}}}  \ \sup_{x,y\in K} \frac{\Big| g_{s}^{N}(x)-g_{s}^{N}(y)\Big|}{|x-y|^{\tilde\eta}} \ \int V(y)  |y|^{\tilde\eta} \  dy \\ 
 & \qquad  \qquad +   \frac{C}{N^{\frac{\beta}{d}}}  \int V(y) |y|     dy ,
\end{align*}
where $K$ is a compact set and $\tilde\eta=\eta-\frac{d}{2}$, where $\eta$ has been defined in Section \ref{subsect compactness}. 
Therefore we have obtained 
\[
 \mathbb E\bigg[\Big|F(x ,\; g_{s}^{N}(x))\nabla  (V^{N}\ast\phi)(x)  
- \big( (F(\cdot \; ,\; g_{s}^{N})\nabla  (V^{N}\ast\phi) \left(\cdot\right)) \ast V^{N} \big)(x) \Big|\bigg]
\leq  \frac{C}{N^{\frac{\tilde\eta \beta}{d}}} . 
\]
Thus,
\begin{align*}
\lim_{N\rightarrow \infty}\int_{0}^{t}  \Big\langle S_{s}%
^{N},F(\cdot \; ,\; g_{s}^{N}) \; &\nabla  (V^{N}\ast\phi) \left(\cdot\right)  \Big\rangle  ds \\
& =\lim_{N\rightarrow \infty}\int_{0}^{t}\Big\langle g_{s}%
^{N},F(\cdot \; ,\; g_{s}^{N})\nabla  (V^{N}\ast\phi) \left(\cdot\right)  \Big\rangle ds.
\\ & =
\lim_{N\rightarrow \infty}\int_{0}^{t}   \int  g_{s}^{N}(x)  F(x,g_{s}^{N}) \nabla  (V^{N}\ast\phi) \left(x\right)  \; dx ds \\
& =  \int  u(s,x)  F(x,u(s,x)) \nabla \phi(x)  \; dx ds
\end{align*}
 where in the last equality we used that $g_s^{N}\rightarrow u$ strongly in 
$\LL^{2}\left( [ 0,T]\; ;\; C(  D)  \right)$. 

\medskip

\noindent \textsc{Step 3}. We  claim that 
\begin{equation}\label{m}
 \lim_{N\rightarrow \infty} \frac{1}{N}\sum_{i=1}^{N} \int_{0}^{t} \int_{\RR^{d}-\{0\}}      \Big\{ (V^{N}\ast\phi) \big(-X_{s_{-}}^{i,N}+ z  \big)  -   (V^{N}\ast\phi) \big(-X_{s_{-}}^{i,N}\big)  \Big\} d\tilde{\mathcal N}^{i}(dsdz)=0.
\end{equation}
First we observe that 
\begin{align}
 \lim_{N\rightarrow \infty} & \EE\bigg[ \bigg|\frac{1}{N}\sum_{i=1}^{N} \int_{0}^{t} \int_{|z|\geq 1}   \Big\{    (V^{N}\ast\phi) \big(-X_{s_{-}}^{i,N}+ z  \big)  -   (V^{N}\ast\phi) \big(-X_{s_{-}}^{i,N}\big) \Big\}  d\tilde{\mathcal N}^{i}(dsdz)\bigg|^{2}\bigg] \notag
\\ 
& =\lim_{N\rightarrow \infty}    \frac{1}{N^{2}} \sum_{i=1}^{N} \int_{0}^{t} \int_{|z|\geq 1} \Big|(V^{N}\ast\phi) \big(-X_{s_{-}}^{i,N}+ z  \big)  -   (V^{N}\ast\phi) \big(-X_{s_{-}}^{i,N}\big)\Big|^{2}  \;   d\nu(z)ds \notag \\
& \leq  \lim_{N\rightarrow \infty}  \frac{C}{N}= 0. \label{m1}
\end{align}
On other hand  we have 
\begin{align}
 \lim_{N\rightarrow \infty} & \EE \bigg[ \bigg|\frac{1}{N}\sum_{i=1}^{N} \int_{0}^{t} \int_{|z|\leq 1}    \Big\{   (V^{N}\ast\phi) \big(-X_{s_{-}}^{i,N}+ z  \big)  -   (V^{N}\ast\phi) \big(-X_{s_{-}}^{i,N}\big)  \Big\} d\tilde{\mathcal N}^{i}(dsdz)\bigg|^{2}\bigg] \notag \\
& =\lim_{N\rightarrow \infty}    \frac{1}{N^{2}} \sum_{i=1}^{N} \int_{0}^{t} \int_{|z|\leq 1} \Big|(V^{N}\ast\phi) \big(-X_{s_{-}}^{i,N}+ z  \big)  -   (V^{N}\ast\phi) \big(-X_{s_{-}}^{i,N}\big)\Big|^{2}  \;   d\nu(z)ds  \notag \\
& \leq   C \lim_{N\rightarrow \infty}    \frac{1}{N^{2}} \sum_{i=1}^{N} \int_{0}^{t} \int_{|z|\leq 1} |z|^{2} \;  d\nu(z) ds   
 \leq C \lim_{N\rightarrow \infty} \frac{1}{N}=0. \label{m2}
\end{align}
From \eqref{m1} and \eqref{m2} we conclude \eqref{m}. Summarizing, we have proved \eqref{limiteq}.

\subsection{Uniqueness of PDE} \label{sec:unique}
In order to make the paper self-contained we present the uniqueness result for the PDE \eqref{PDEintr}. 
We also refer  the reader to  \cite{ali,Dro,Karc}.

\begin{theorem}
\label{Thm PDE}. 
There is at most one weak solution of equation \eqref{PDEintr} in  $\LL^{2}\big( [ 0,T]\; ;\;\HH^{\eta} \big)$ with $\eta> \frac{d}{2}$.   
\end{theorem}

\begin{proof}
Let $u^{1},u^{2}$ be two weak solutions of the equation \eqref{PDEintr}
 with the same initial condition $u_{0}$. Let $\{\rho_{\varepsilon}%
(x)\}_{\varepsilon}$ be a family of standard symmetric mollifiers. For any
$\varepsilon>0$ and $x\in\mathbb{R}^{d}$ we can use $\rho_{\varepsilon
}(x-\cdot)$ as test function in the equation \eqref{limiteq}.  Set
$u_{\varepsilon}^{i}(t,x)=u^{i}(t,\cdot)\ast\rho_{\varepsilon}(\cdot)(x)$ for
$i=1,2$. Then we have%
\[
u_{\varepsilon}^{i}(t,x)=(u_{0}\ast\rho_{\varepsilon})(x)+\int_{0}^{t} \mathcal{L}
u_{\varepsilon}^{i}(s,x)\,ds+\int_{0}^{t}\big(\nabla \rho_{\varepsilon}\ast u^{i} F(\cdot,u^{i})\big)(s,x)  \,ds.
\]
Writing this identity in mild form we obtain (we write $u^{i}\left(  t\right)
$ for the function $u^{i}\left(  s,\cdot\right)  $ and $S(t)$ for $e^{tL}$)%
\[
u_{\varepsilon}^{i}(t)=S(t)(u_{0}\ast\rho_{\varepsilon})+\int_{0}%
^{t}S(t-s)\ \big(\nabla \rho_{\varepsilon}\ast u^{i} F(\cdot,u^{i})\big)  ds.
\]
The function $U=u^{1}-u^{2}$ satisfies 
\[
\rho_{\varepsilon}\ast U(t)=\int_{0}^{t} \nabla S(t-s) \big( \rho_{\varepsilon}\ast \big[ u^{1} F(\cdot,u^{1})- u^{2} F(\cdot,u^{2})\big] \big) \,ds.
\]
Thus we obtain 
\[
\|\rho_{\varepsilon}\ast U(t)\|_{\LL^{2}}\leq 
\int_{0}^{t} \Big\|  \nabla S(t-s) \big( \rho_{\varepsilon}\ast \big[ u^{1} F(\cdot,u^{1})- u^{2} F(\cdot,u^{2})\big] \big) \Big\|_{\LL^{2}} \,ds.
\]
Using the proposition \ref{prop} we have 
\[
\|\rho_{\varepsilon}\ast U(t)\|_{\LL^{2}}\leq 
\int_{0}^{t}  \frac{1}{(t-s)^{\frac{1}{2}}}  \Big\|   \rho_{\varepsilon}\ast \big[ u^{1} F(\cdot,u^{1})- u^{2} F(\cdot,u^{2})\big]  \Big\|_{\LL^{2}} \,ds.
\]
Taking the limit as $\varepsilon\rightarrow0$ we  arrive 
\[
\| U(t)\|_{\LL^{2}}\leq 
\int_{0}^{t}  \frac{1}{(t-s)^{\frac{1}{2}}}  \Big\|  \big[ u^{1} F(\cdot,u^{1})- u^{2} F(\cdot,u^{2})\big]  \Big\|_{\LL^{2}} \,ds.
\]
By easy calculation we have 
\[
\| U(t)\|_{\LL^{2}}\leq 
\int_{0}^{t}  \frac{1}{(t-s)^{\frac{1}{2}}}  \big\|  U  F(\cdot,u^{1})  \big\|_{\LL^{2}}  +  \big\|  u^{2} ( F(\cdot,u^{1}) - F(\cdot,u^{1}))  \big\|_{\LL^{2}} \,ds.
\]
Notice that the function $F$ is globally Lipschitz and bounded. It follows
\[
\| U(t)\|_{\LL^{2}}
\leq \int_{0}^{t}    \frac{1}{(t-s)^{\frac{1}{2}}} \big( C   \|  U   \|_{\LL^{2}} +  \|  u^{2}  U  \|_{\LL^{2}} \big)  \,ds.
\]
By hypothesis $u^{2} \in \LL^{2}(  [0,T]\; ;\; \HH^{\eta}) $ with $\eta> \frac{d}{2}$. Then by the Sobolev embeddings (see \cite[Section 2.8.1]{Triebel}),
we have $u^{2}\in \LL^{2}\big(  [0,T]\; ;\; C_{b}( \mathbb{R}^{d}) \big) $. It follows that 
\[
\| U(t)\|_{\LL^{2}}
\leq \int_{0}^{t}    \frac{1}{(t-s)^{\frac{1}{2}}}\big( C   \|  U   \|_{\LL^{2}}  +  \| u^{2}\|_{\LL^{\infty}}\;  \| U \|_{\LL^{2}} \big) \,ds.
\]
By Gronwall's Lemma we conclude $U=0$.
\end{proof}

\subsection{Convergence in probability}

\begin{corollary}\label{coroconveP2} The sequence $\{g^{N}\}_{N\in\NN}$ converges in probability to $u$.
\end{corollary}

\begin{proof} 
 We denote  the joint law of $ (g^{N} ,g^{M} )$  by $\nu^{N,M}$. Similarly  to the proof of tightness for $g^{N}$  we have that the family $\{\nu^{N,M}\}$ is tight in    $Y \times  Y$.

Let us take any subsequence $ \nu^{N_k,M_k} $. By  Prohorov's  theorem, it is
relatively weakly compact hence it contains a weakly convergent subsequence. Without
loss of generality we may assume that the original sequence $\{\nu^{N,M} \}$ itself converges
weakly to a measure $\nu$. According to the Skorokhod immersion  theorem, we
infer the existence of a probability space  $\big(  \bar{\Omega}, \bar{\mathcal{F}}, \bar{\mathbb{P}} \big)$   with a sequence of random variables
$(\bar{g}^{N}, \bar{g}^{M})$ converging almost surely in $Y \times Y $ to random variable  $   (\bar{u}, \check{u})$ and
 the laws of $(\bar{g}^{N}, \bar{g}^{M}) $  and $ (\bar{u}, \check{u})$
  under $ \bar{\mathbb{P}}$  coincide with  $  \nu^{N,M} $ and   $\nu$, respectively.

	Analogously, it can be applied to both  $\bar{g}^{N}$   and $ \bar{g}^{M}$  in order
to show that  $\bar{u}$ and $\check{u}$    are  two solutions of the PDE \eqref{limiteq}. By the uniqueness property of the solutions to  \eqref{limiteq} we have  $\bar{u}=\check{u} $. 
 Therefore 
	$$
	\nu  \big((x,y)\in Y\times Y \; : \; x=y\big)= \bar{\mathbb{P}} (\bar{u}=\check{u})=1. 
$$
Now, we have all in hands to apply  Gyongy-Krylov's characterization
of convergence in probability (Lemma \ref{GK}).  It implies that the original
sequence is  defined on the initial probability space converges in probability
in the topology of $Y$ to a random variable $\mu$.
\end{proof}

%%%%%%%%%%%%%%%%%%%%%%%%
\section*{Acknowledgements}
%%%%%%%%%%%%%%%%%%%%%%%%

C.O. is partially supported by  CNPq
through the grant 460713/2014-0 and FAPESP by the grants 2015/04723-2 and 2015/07278-0. 
This work benefited from the support of the project EDNHS ANR-14-CE25-0011 of the French National Research Agency (ANR). The work of M.S. was also supported  by the Labex CEMPI (ANR-11-LABX-0007-01).


\begin{thebibliography}{99}   


\bibitem{ali}
N. Alibaud, Entropy formulation for fractal conservation laws, \emph{J. Evol. Equ.} {\bf 7}, 145--175 (2007). 

\bibitem{Al} M. Alfaro, J. Droniou, General fractal conservation laws arising from a model of detonations in gases, {\it Applied Mathematics Research eXpress} Oxford University Press (OUP), 127--151 (2012).

\bibitem{Andre}
 F. Andreu, J. M. Mazon, J. D. Rossi, J. Toledo, {\it Nonlocal Diffusion Problems}, AMS
Mathematical Surveys and Monographs {\bf 165}, 2010.

\bibitem{App}
D. Applebaum, {\it L\'{e}vy Processes and Stochastic Calculus},  Cambridge University Press, 2nd. ed., 2009.


\bibitem{Biler}  P. Biler,  T. Funaki, W.A.  Woyczynski,  Interacting particle approximations for
nonlocal quadratic evolution problems, {\it Probability and Mathematical Statistics}  {\bf 2}, 267--286 (1999).

\bibitem{capasso} V. Capasso, D. Bakstein, \textit{An introduction to continuous-time stochastic processes : theory, models, and applications to finance, biology and medicine}, Series Modeling and Simulation in Science, Engineering and Technology, Birkh\"auser Basel, Springer Science+Business Media New York, 2015.



\bibitem{Cla}
P. Clavin, Instabilities and non-linear patterns of over driven detonations in gases, {\it Non-linear PDE’s in Condensed Matter and Reactive
Flows}, Kluwer,   49--97 (2002).


%
\bibitem {DaPrZab}G. Da Prato, J. Zabczyk, \textit{Stochastic Equations in
Infinite Dimensions}, Cambridge Univ. Press, Cambridge 1992.


\bibitem{Dro} J. Droniou, C.  Imbert, Fractal first-order partial differential equations
 {\it Arch. Ration. Mech. Anal.} {\bf 182}, 299--331 (2006).

\bibitem{FlaGat} F. Flandoli, D. Gatarek, Martingale and stationary solutions
for stochastic Navier-Stokes equations, \textit{Prob. Th. Relat. Fields}
\textbf{102}, 367--391 (1995).



\bibitem{flan} F. Flandoli, M. Leimbach,  C. Olivera, Uniform approximation of FKPP equation by stochastic particle systems,
  ArXiv:1604.03055 (2016).



\bibitem{Gyon}
I. Gyongy, N. Krylov,  Existence of strong solutions for It\^o stochastic equations, 
via approximations, {\it Prob. Th. Relat. Fields} {\bf 105},  143--158 (1996).



\bibitem{Ikeda}
 N. Ikeda, S. Watanabe, {\it Stochastic Differential Equations and Diffusion Processes}, North Holland Pub. Co. Amsterdam, Oxford, New York, 1981.





\bibitem{Mela2} B. Jourdain, S. M\'el\'eard,  W.A. Woyczynski,  A probabilistic approach for non-linear
equations involving fractional Laplacian and singular operator, {\it  Potential Analysis} {\bf 23}(1), 55--81 (2005).


\bibitem{Mela3}
B. Jourdain, S. M\'el\'eard, W.A. Woyczynski,  Probabilistic approximation and inviscid limits for 1-D fractional conservation laws, {\it Bernoulli} {\bf 11}(4), 689--714 (2005).



\bibitem{Mela}
B Jourdain, S. M\'el\'eard,    Propagation of chaos and fluctuations for a moderate model with smooth initial data, {\it 
Ann. IHP (B) Probabilit\'es et Statistiques} {\bf 34}(6),  727--766  (1998).

\bibitem{Karc}
 G. Karch, {\it Non-linear evolution equations with anomalous diffusion},  Qualitative properties
of solutions to partial differential equations, Jindich Neas Cent. Math. Model, Lect. Notes, 5, Matfyzpress, Prague, 2--68, 2009.




\bibitem{Kunita}
H. Kunita,  Stochastic differential equations based on L\'{e}vy processes and stochastic flows of diffeomorphisms,
{\it Real and Stoch. Anal}, 305--373 (2004).


\bibitem {Lions} J. L. Lions, \textit{Quelques M\'ethodes de R\'esolution des
Probl\`emes aux Limites non Lin\'eaires}, Dunod, Paris, 1969.

\bibitem {MelRoelly} S. M\'el\'eard, S. Roelly-Coppoletta, A propagation of chaos
result for a system of particles with moderate interaction, \textit{Stochastic
Processes Appl}. \textbf{26}, 317--332 (1987).


\bibitem{Oel1} K. Oelschl\"{a}ger, A law of large numbers for moderately
interacting diffusion processes, \textit{Zeitschrift fur Wahrsch. Verwandte
Gebiete} \textbf{69}, 279--322 (1985).

\bibitem {Priola} E. Priola, Pathwise uniqueness for singular SDEs drive
n by stable processes,   {\it  Osaka J. Math.} {\bf 49}(2), 421--447  (2012).

%\bibitem {Pa} A. Pazy, \textit{semi-groups of Linear Operators and Applications
%to Partial Differential Equations}, Springer, 1983.


\bibitem{Sato}
 K. Sato,  {\it L\'{e}vy Processes and Infinitely Divisible Distributions}, 
 Cambridge Studies in Advanced Mathematics 2nd ed, 2013.

\bibitem{Stein}
 E.M. Stein,  {\it Singular integrals and differentiability properties of functions}, Princeton University Press, 1970.
%
\bibitem{Triebel} H. Triebel, {\it Interpolation Theory, Function Spaces,
Differential Operators}, North-Holland, Amsterdam, 1978.

\bibitem{Var} S.R.S. Varadhan, Scaling limits for interacting diffusions, \textit{Commun. Math. Phys.} \textbf{135}, 313--353 (1991).


\bibitem{Vaz} J.L. Vazquez,
Recent progress in the theory of non-linear diffusion with fractional Laplacian operators, 
{\it Discrete  Continuous Dynamical Systems - Series S} {\bf 7}, 857--885 (2014). 

\bibitem{Zhang2} X. Zhang, $L^{p}$-maximal regularity of nonlocal parabolic equation and applications, \emph{Annales de l'Institut Henri Poincare (C) Non Linear Analysis}, {\bf 30}(4), 573--614 (2013).


\bibitem {Woy}
 W. Woyczynski, {\it L\'evy processes in the physical sciences}, Birkh\"auser, Boston, 241--266 (2001).





\end{thebibliography}
\end{document}